\documentclass[smallheadings,headsepline,11pt,a4paper]{article}

\usepackage[applemac]{inputenc} 
\usepackage{hyperref} 
\usepackage[english]{babel}
\usepackage{dsfont}
\usepackage{amsmath}
\usepackage{amssymb}
\usepackage{amsthm}
\title{Hyperclass Forcing in Morse-Kelley Class Theory}
\author{Carolin Antos, Sy-David Friedman\footnote{The first author wants to thank the Austrian Academy of Sciences for their generous support through their Doctoral Fellowship Program. Both authors are grateful to the John Templeton Foundation for its generous support through Grant ID 35216, which supported the preparation of this article.}\\
Kurt G\"odel Research Center for Mathematical Logic\\
University of Vienna}

\date{}
\begin{document}

\newtheorem{defi}{Definition} 
\newtheorem{theo}[defi]{Theorem}
\newtheorem{cor}[defi]{Corollary}
\newtheorem{prop}[defi]{Proposition}
\newtheorem{lem}[defi]{Lemma}
\newtheorem{rem}[defi]{Remark}
\newtheorem{claim}[defi]{Claim}

\maketitle

\begin{abstract}
In this article we introduce and study hyperclass-forcing  (where the conditions of the forcing notion are themselves classes) in the context of an extension of Morse-Kelley class theory, called MK$^{**}$.
We define this forcing by using a symmetry between MK$^{**}$ models and models  of ZFC$^-$ plus there exists a strongly inaccessible cardinal (called SetMK$^{**}$). We develop a coding between $\beta$-models $\mathcal{M}$ of MK$^{**}$ and transitive models $M^+$ of SetMK$^{**}$ which will allow us to go from $\mathcal{M}$ to $M^+$ and vice versa. So instead of forcing with a hyperclass in MK$^{**}$ we can force over the corresponding SetMK$^{**}$ model with a class of conditions. For class-forcing to work in the context of ZFC$^-$ we show that the SetMK$^{**}$ model $M^+$ can be forced to look like $L_{\kappa^*}[X]$, where $\kappa^*$ is the height of $M^+$, $\kappa$ strongly inaccessible in $M^+$ and $X\subseteq\kappa$. Over such a model we can apply definable class forcing and we arrive at an extension of $M^+$ from which we can go back to the corresponding $\beta$-model of MK$^{**}$, which will in turn be an extension of the original $\mathcal{M}$. 
Our main result combines hyperclass forcing with coding methods of \cite{beller1982coding} and \cite{Friedman:2000} to show that every $\beta$-model of MK$^{**}$ can be extended to a minimal such model of MK$^{**}$ with the same ordinals. A simpler version of the proof also provides a new and analogous minimality result for models of second-order arithmetic.

\end{abstract}

\section{Introduction}

When considering forcing notions with respect to their size, there are three different types: the original version of forcing, where the forcing notion is a set, called set forcing; forcing in ZFC, where the forcing notion is a class, called definable class forcing and  class forcing in Morse-Kelley class theory (MK). In this article we consider a fourth type which we call definable hyperclass forcing and give applications for this forcing in the context of Morse-Kelley class theory, where hyperclass forcing denotes a forcing with class conditions. We will define hyperclass forcing indirectly by using a correspondence between certain models of MK and models of a version of ZFC$^-$ (minus PowerSet) and show that we can define definable hyperclass forcing by going to the related ZFC$^-$ model and using definable class forcing there.

Two problems arise when considering definable class forcing in ZFC: the forcing relation might not be definable in the ground model and the extension might not preserve the axioms. As an example consider $Col(\omega, ORD)$ with conditions $p: n\to Ord$ for $n\in\omega$ which adds a cofinal sequence of length $\omega$ in the ordinals. Here Replacement fails\footnote{A detailed analyses on how even the Definability Lemma for class forcings can fail can be found in \cite{holy_class}.}. These problems were addressed in a general way by the second author in \cite{Friedman:2000}) where class forcings are presented which are definable (with parameters) over a model $\langle M, A\rangle$ where $M$ is a transitive model of ZFC, $A\subseteq M$ and Replacement holds in $M$ for formulas mentioning $A$ as a unary predicate. Two properties of the forcing notion are introduced, pretameness and tameness and it is shown that for a pretame forcing notion the Definability Lemma holds and Replacement is preserved and that tameness (which is a strengthening of pretameness) is equivalent to the preservation of the Power Set axiom. In this article we  will adjust this approach to definable class forcing in ZFC$^-$. Pretameness is defined as follows:

\begin{defi}
A forcing notion $P$ is pretame iff whenever $\langle D_i\vert i\in a\rangle$, $a\in M$, is an $\langle M, A\rangle$-definable sequence of dense classes and $p\in P$ then there is $q\leq p$ and $\langle d_i\vert i\in a\rangle\in M$ such that $d_i\subseteq D_i$ and $d_i$ is predense $\leq q$ for each $i$.
\end{defi}

For definable hyperclass forcing we will work in the context of Morse-Kelley class theory, by which we mean a theory with a two-sorted language, i.e. the object are sets and classes and we have corresponding quantifiers for each type of object. We denote the classes by upper case letters and sets by lower case letters, the same will hold for class-names and set-names and so on. Hence atomic formulas for the $\in$-relation are of the form ``$x\in X$'' where $x$ is a set-variable and $X$ is a set- or class-variable. The models $\mathcal{M}$ of MK  are of the form $\langle M, \in, \mathcal{C}\rangle$, where $M$ is a transitive model of ZFC, $\mathcal{C}$ the family of classes of $\mathcal{M}$ (i.e. every element of $\mathcal{C}$ is a subset of $M$) and $\in$ is the standard $\in$ relation (from now on we will omit mentioning this relation). We use the following axiomatization of MK:

\begin{itemize}
\item[A)] Set Axioms:
\begin{enumerate}
\item Extensionality for sets: $\forall x \forall y (\forall z\,(z\in x\leftrightarrow z\in y)\to x=y)$.
\item Pairing: For any sets $x$ and $y$ there is a set $\{x,y\}$.
\item Infinity: There is an infinite set.
\item Union: For every set $x$ the set $\bigcup x$ exists.
\item Power set: For every set $x$ the power set $P(x)$ of $x$ exists.
\end{enumerate}
\item[B)] Class Axioms:
\begin{enumerate}
\item Foundation: Every nonempty class has an $\in$-minimal element.  
\item Extensionality for classes: $\forall z\,(z\in X\leftrightarrow z\in Y)\to X=Y$.
\item Replacement: If a class $F$ is a function and $x$ is a set, then $\{ F(z): z\in x\}$ is a set.
\item Class-Comprehension: \begin{equation*}
\forall X_1\ldots\forall X_n\exists Y\; Y=\{x: \varphi(x,X_1,\ldots, X_n)\}
\end{equation*}
where $\varphi$ is a formula containing class parameters in which quantification over both sets and classes are allowed.
\item Global Choice: There exists a global class well-ordering of the universe of sets.
\end{enumerate}
\end{itemize}

Class forcing in MK was defined by the first author in \cite{antosclassforcing}. Here the Definability Lemma holds for unrestricted forcing notions, but for the preservation of the axioms we still need pretameness and tameness.

The structure of this article will be as follows: First, we define the correspondence between certain models of a version of MK and ZFC$^-$ and show that this correspondence is indeed a coding between a variant of MK and certain models of the ZFC$^-$ which allows us to go back and forth between them. Then we define definable hyperclass forcing and show how the problems of definable class forcing in the setting of ZFC$^-$ can be handled. We conclude the chapter by giving an example of definable hyperclass forcing by showing that every $\beta$-model of a variant of MK can be extended to a minimal $\beta$-model of the same variant of MK with the same ordinals.

\section{Coding between MK$^*$ and SetMK$^*$}\label{codingsection}

In the context of ZFC we can talk about definable class forcings as done in \cite{Friedman:2000}, where we deal directly with the class forcing notion as it is definable from a class predicate. Here we want to develop a way of defining definable hyperclass forcings in MK, i.e. forcings with class conditions, but we will choose an indirect approach, which will allow us to reduce the technical problems as much as possible to the context of definable class forcing. So instead of talking directly about hyperclasses, we will use a correspondence between models of a variant of MK (called MK$^*$) and models of a variant of ZFC$^-$ (called SetMK$^*$). We get an idea of how such a model of SetMK$^*$ looks by considering the following model of MK: $\langle V_{\kappa}, V_{\kappa+1}\rangle$ where $\kappa$ is strongly inaccessible. Similar to this model we will show how to define a model of SetMK$^*$ with a strongly inaccessible cardinal $\kappa$ which is the largest cardinal such that the sets of the MK$^*$ model are elements of $V_{\kappa}$ and the classes are elements of $V_{\kappa^*}$, where $\kappa^*$ is the height of the SetMK$^*$ model. We will then force over such a model with a definable class forcing which will give us an extension of the SetMK$^*$ model. From this extension we can then go back to a model of MK$^*$ and this is the definable hyperclass-generic extensions of the original MK$^*$ model.

\begin{picture}(70,60)
\put(90,7){\vector(1,0){55}}
\put(75,15){\vector(0,1){20}}
\put(155,35){\vector(0,-1){20}}
\put(90,43){\vector(1,0){55}}
\put(70,0){$\mathcal{M}$}
\put(70,40){$M^+$}
\put(150,40){$M^+[G]$}
\put(150,0){$\mathcal{M}[G]$}
\put(10,0){In MK$^*$:}
\put(10,40){In SetMK$^*$:}
\put(93,46){$\scriptstyle{def.\, class}$}
\put(93,36){$\scriptstyle{forcing}$}
\put(93,10){$\scriptstyle{def.\, hyperclass}$}
\put(93,0){$\scriptstyle{forcing}$}
\end{picture}\\

 In the following we will describe how we can go from MK$^*$ to SetMK$^*$ and vice versa and show that the basic properties of class forcing over a model of SetMK$^*$ hold. Then we give an application of definable hyperclass forcing regarding minimal models of MK$^{**}$. 

But before we develop the relation between these models further we will impose a restriction on the models we are considering. 
\begin{defi}
A model $\mathcal{M}$ of Morse-Kelley class theory is a $\beta$-model of MK if a class is well-founded in $\mathcal{M}$ if and only if it is true that the class is well-founded. 
\end{defi}
We introduce this restriction for two reasons: First, we will define a coding which allows us to go from a $\beta$-model of MK$^*$ to a transitive model of SetMK$^*$ and this coding only works in the intended way if we know that every well-founded class in the model is really well-founded (see section \ref{codingsection}). Secondly we will prove a theorem about minimal models and such a notion only makes sense if we work with minimal $\beta$-models. So from now on, we will always talk about $\beta$-models of (variants of) MK.

The associated model of set theory will be a model of ZFC$^-$ (i.e. minus the Power Set Axiom) where we understand such a model to include the Collection (or Bounding) Principle\footnote{Note that in ZFC minus Power Set the Bounding Principle does not follow from Replacement. This is used in \cite{Zarach:1982}, where he showed that in ZF$^-$ the different formulations of the Axiom of Choice are not equivalent. As for MK, work done in \cite{hamgit} shows that for example ultrapower constructions don't work without first adding a version of Class-Bounding.}. To ensure this we have to add the Class-Bounding Principle, a ``class version'' of the Bounding Principle, and we call the resulting axiomatic system MK$^*$:
\begin{defi}
The axioms of MK$^*$ consist of the axioms of MK plus the Class-Bounding Axiom
\begin{equation*}
\forall x\,\exists A\, \varphi(x,A)\to\exists B\,\forall x\,\exists y\, \varphi(x,(B)_{y})
\end{equation*}
where $(B)_{y}=\{z\,\vert\, (y,z)\in B\}$.\\
\end{defi}  

Note that as we have Global Choice, this is equivalent to AC$_{\infty}$: 
\begin{equation*}
\forall x\,\exists A\, \varphi(x,A)\to\exists B\,\forall x\, \varphi(x,(B)_{x}).
\end{equation*}

Equivalently, SetMK$^*$ will include the set version of Bounding (here called Set-Bounding):
\begin{equation*}
\forall x\in a \,\exists y \,\varphi(x,y)\to \exists b\,\forall x\in a\,\exists y\in b\,\varphi(x,y)
\end{equation*}
As we will show in the proof of Theorem \ref{maintheorem} and the proof of Theorem \ref{conversemain}, Set-Bounding in SetMK$^*$ follows from Class-Bounding in MK$^*$ and vice versa.

We are now going to show how to translate the theory of MK$^*$ to a first-order set theory SetMK$^*$. The axioms of SetMK$^*$ are:
\begin{enumerate}
\item ZFC$^-$ (including Set-Bounding).  
\item There is a strongly inaccessible cardinal $\kappa$.
\item Every set can be mapped injectively into $\kappa$.
\end{enumerate}
We can construct a transitive model $M^+$ of SetMK$^*$ out of any $\beta$-model $(M,\mathcal{C})$ of MK$^*$ by taking all sets which are coded by a pair $(M_0, R)$, where $M_0$ belongs to $\mathcal{C}$ and $R$ is a binary relation within $\mathcal{C}$. We will show that $M^+$ is the unique model of SetMK$^*$ with largest cardinal $\kappa$ such that $M=V_{\kappa}^{M^+}$ and the elements of $\mathcal{C}$ are the subsets of $M$ in $M^+$. 

To describe the coding between SetMK$^*$ and MK$^*$ we will define what a coding pair $(M_0, R)$ is and what it means for a coding pair $(M_0, R)$ to code a set $x$ in a model of SetMK$^*$.

\begin{defi}\label{codingdefi}
A pair $(M_0,R)$ is a coding pair in the $\beta$-model $\mathcal{M}=(M, \mathcal{C})$ if $M_0$ is an element of $\mathcal{C}$ with a distinguished element $a$,
$R\in\mathcal{C}$ and $R$ is a binary relation on $M_0$ with the following properties:
\begin{itemize}
\item[a)] $\forall z\in M_0\,\exists ! n$ such that $z$ has $R$-distance $n$ from $a$, i.e. there is an $R$-chain $(z R z_{n-1} R \ldots R z_{1} R a)$,
\item[b)] if $x, y, z\in M_0$ with $y\neq z$, $y R x$, $z R x$ then $(M_0, R)\restriction y$ is not isomorphic to $(M_0, R)\restriction z$, where $(M_0, R)\restriction y$ denotes the $R$-transitive closure below $y$ (i.e. $y$ together with all elements which are connected to $y$ via an $R$-chain), respectively for $z$,
\item[c)] if $y,z\in M_{0}$  are on level $n$ (i.e. have the same $R$-distance $n$ from $a$) and $y\neq z$ then $v R y \,\to\, \neg(v R z)$,
\item[d)] $R$ is well-founded.
\end{itemize}
\end{defi}

Note that in the definition of the codes in $(M,\mathcal{C})$ we need the assumption that $(M,\mathcal{C})$ is a $\beta$-model as for a class to code a set in $M^+$ it has to be well-founded not only in the MK model but ``in the real world''.

The meaning of the definition becomes clearer when we view the coding pair as a tree $T$ whose nodes are exactly the distinct elements of $M_0$, the top node is $a$ and $R$ is the extension relation of the tree. A tree $T'$ with top node $a'$ is a subtree of $T$ if $a'$ is a node of $T$ and $T'$ contains all $T$-nodes (not only immediately) below $a'$.  If $T'$ is a subtree of $T$ such that $a'$ lies directly below $a$ then $T'$ is called a direct subtree of $T$. Then property $b)$ states that for every node $x$ distinct direct subtrees are not isomorphic and property $c)$ implies that the trees below two distinct points on the same level are disjoint (and not only on the next level).

The idea behind the coding pairs is, that every coding pair will define a unique set $x$ in the SetMK$^*$ model. Note that at the same time every $x$ in $M^+$ can correspond to different coding pairs in $\mathcal{M}$. 

In the following we will give some intuition on what such a correspondence between coding pairs in $\mathcal{M}$ and sets in $M^+$ should look like: Every $x\in M^+$ is coded by a tree $T_x$ where $x$ is associated to the top node $a_x$ of $T_x$, the elements $y\in x$ are associated to the nodes on the first level below $a_x$ so that every node on this level gives rise to a subtree $T_y$ which codes $y$ so that the elements of $y$ are associated to the nodes on the second level below $a_x$ and so on:\\
\\
\\
\\

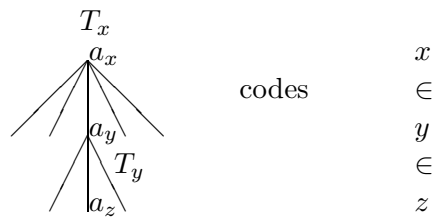
\setlength{\unitlength}{1cm} 
\begin{picture}(1,1)
\put(2.9,1.4){$T_x$}

\put(3,1){\line(0,-1){1}}
\put(3,1){\line(-1,-1){1}}
\put(3,1){\line(1,-1){1}}
\put(3,1){\line(-1,-2){0.5}}
\put(3,1){\line(1,-2){0.5}}
\put(3,0){\line(-1,-2){0.5}}
\put(3,0){\line(1,-2){0.5}}
\put(3,0){\line(0,-1){1}}
\put(3.02,1){$a_{x}$}
\put(3.02,0){$a_y$}
\put(3.02,-1){$a_z$}
\put(3.35,-0.5){$T_y$}

\put(5,0.5){codes}
\put(7.3,1){$x$}
\put(7.3,0.5){$\in$}
\put(7.3,0){$y$}
\put(7.3,-0.5){$\in$}
\put(7.3,-1){$z$}

\end{picture}\\
\\
\\
\\
Note that there are only countably many levels but a level can have class many elements. If two elements $a_y, a_z$ have the same $R_x$ predecessor (i.e. are connected to the same node on the previous level) their subtrees $T_y, T_z$ will never be isomorphic and therefore don't code the same element of $M^+$ (by property $b)$ of Definition \ref{codingdefi}). But it can happen that there are isomorphic subtrees on different levels or on the same level but not connected to the same node on the level above. This can be made clear in the following two examples: First let $y\in x$, $v\in y$ and $w\in y$ and $v\in w$. Then there are two isomorphic trees $T_v$ and $T'_v$ both coding $v$ but on different levels:\\
\\
\\

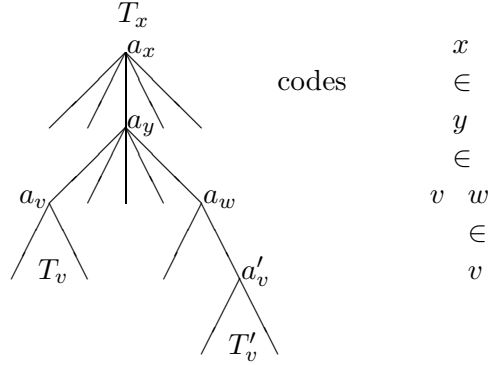
\setlength{\unitlength}{1cm} 
\begin{picture}(1,1)
\put(2.9,1.4){$T_x$}

\put(3,1){\line(0,-1){1}}
\put(3,1){\line(-1,-1){1}}
\put(3,1){\line(1,-1){1}}
\put(3,1){\line(-1,-2){0.5}}
\put(3,1){\line(1,-2){0.5}}
\put(3,0){\line(-1,-1){1}}
\put(3,0){\line(1,-1){1}}
\put(3,0){\line(0,-1){1}}
\put(3,0){\line(-1,-2){0.5}}
\put(3,0){\line(1,-2){0.5}}
\put(2,-1){\line(-1,-2){0.5}}
\put(2,-1){\line(1,-2){0.5}}
\put(4,-1){\line(-1,-2){0.5}}
\put(4,-1){\line(1,-2){0.5}}
\put(4.5,-2){\line(-1,-2){0.5}}
\put(4.5,-2){\line(1,-2){0.5}}
\put(3.02,1){$a_{x}$}
\put(3.02,0){$a_y$}
\put(1.6,-1){$a_v$}
\put(4.02,-1){$a_w$}
\put(4.52,-2){$a'_v$}
\put(1.85,-2){$T_v$}
\put(4.35,-3){$T'_v$}

\put(5,0.5){codes}
\put(7.3,1){$x$}
\put(7.3,0.5){$\in$}
\put(7.3,0){$y$}
\put(7.3,-0.5){$\in$}
\put(7,-1){$v$}
\put(7.5,-1){$w$}
\put(7.5,-1.5){$\in$}
\put(7.5,-2){$v$}

\end{picture}\\
\\
\\
\\
\\
\\
\\

Secondly let $v\in w$, $v\in y$ and $w,y\in x$. Again there are two isomorphic trees $T_v$ and $T'_v$ coding $v$ but this time on the same level:\\
\\
\\

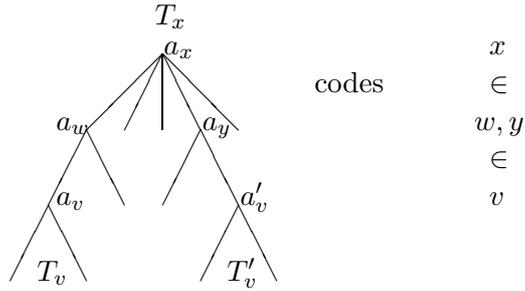
\setlength{\unitlength}{1cm} 
\begin{picture}(1,1)
\put(2.9,1.4){$T_x$}

\put(3,1){\line(0,-1){1}}
\put(3,1){\line(-1,-1){1}}
\put(3,1){\line(1,-1){1}}
\put(3,1){\line(-1,-2){0.5}}
\put(3,1){\line(1,-2){0.5}}

\put(3.5,0){\line(-1,-2){0.5}}
\put(3.5,0){\line(1,-2){0.5}}
\put(2,0){\line(-1,-2){0.5}}
\put(2,0){\line(1,-2){0.5}}
\put(1.5,-1){\line(-1,-2){0.5}}
\put(1.5,-1){\line(1,-2){0.5}}
\put(4,-1){\line(-1,-2){0.5}}
\put(4,-1){\line(1,-2){0.5}}
\put(3.02,1){$a_{x}$}
\put(3.52,0){$a_y$}
\put(1.6,-1){$a_v$}
\put(1.6,0){$a_w$}
\put(4.02,-1){$a'_v$}
\put(1.35,-2){$T_v$}
\put(3.85,-2){$T'_v$}

\put(5,0.5){codes}
\put(7.3,1){$x$}
\put(7.3,0.5){$\in$}
\put(7.1,0){$w,y$}
\put(7.3,-0.5){$\in$}
\put(7.3,-1){$v$}

\end{picture}\\
\\
\\
\\
\\

To show that Definition \ref{codingdefi} indeed defines a coding, we have to show that there is a correspondence between $x$ and its coding pair. As we want to include non-transitive sets we will work with $(TC(\{x\}),\in)$ (note that we used the transitive closure of $\{x\}$ rather than the transitive closure of $x$ as the transitive closure of two different sets could be the same). As we have seen, the coding tree will have a lot of isomorphic subtrees, for example many different pairs $(a_i,\{\})$ coding the empty set. So the tree $T_x$ itself will not be isomorphic to $(TC(\{x\}),\in)$ and we will have to collapse $(M_x, R_x)$ to a structure $(M_x, R_x)/\approx$ in which we have identified all these isomorphic subtrees. We define this quotient of the coding pair in the following way: 
\begin{defi}\label{quotientstructure}
For a coding pair $(M_0, R)$, let $[a]=\{ b\in M_0\,\vert\,(M_0, R)\restriction b$ isomorphic to $(M_0, R)\restriction a\}$ be the equivalence class of all the top nodes of subtrees of the coding tree $T$ which are isomorphic to the subtree $T_a$ (here $(M_0, R)\restriction b$ denotes the ``sub-coding pair'' which is the subtree $T_b$ as detailed in Definition \ref{codingdefi}). By Global Choice let $\tilde{a}$ be a fixed representative of this class. Then let $\tilde{M}_0=\{\tilde{a}\,\vert\, [a]\text{ for all } a\in M_0\}$ and define the relation $\tilde{R}$ as follows: $\tilde{a}\tilde{R}\tilde{b}$ iff $\exists a_0, b_0$ such that $a_0\in[a]$ and $b_0\in[b]$ and $a_0 R b_0$.
\end{defi}
Note that if $a_0\approx a_1$ and $b_0 R a_0$ then there is $b_1$ with $b_1 R a_1$ such that $b_0\approx b_1$ as the isomorphism between $T_{a_0}$ and $T_{a_1}$ will restrict to the trees $T_{b_0}$ and $T_{b_1}$. 

The following example shows how this quotient structure looks for a possible coding tree of the set $3$:

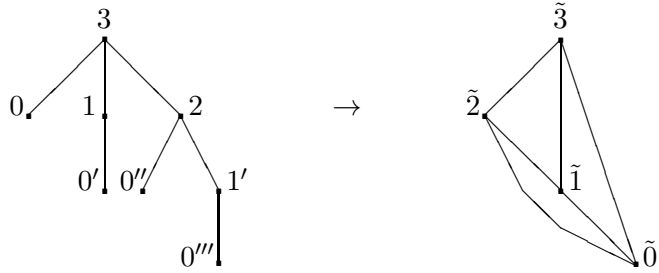
\setlength{\unitlength}{1cm} 
\begin{picture}(1,1)

\put(3,0){\line(-1,-1){1}}
\put(3,0){\line(1,-1){1}}
\put(3,0){\line(0,-1){1}}
\put(3,-1){\line(0,-1){1}}
\put(4,-1){\line(-1,-2){0.5}}
\put(4,-1){\line(1,-2){0.5}}
\put(4.5,-2){\line(0,-1){1}}

\put(2.9,0.15){$3$}
\put(2.95,-0.03){$\centerdot$}
\put(1.75,-1){$0$}
\put(1.95,-1.05){$\centerdot$}
\put(2.7,-1){$1$}
\put(2.95,-1.05){$\centerdot$}
\put(2.65,-2){$0'$}
\put(2.95,-2.05){$\centerdot$}
\put(4.1,-1){$2$}
\put(3.95,-1.05){$\centerdot$}
\put(4.6,-2){$1'$}
\put(4.45,-2.05){$\centerdot$}
\put(3.2,-2){$0''$}
\put(3.45,-2.05){$\centerdot$}
\put(4,-3){$0'''$}
\put(4.45,-3){$\centerdot$}

\put(6,-1){$\to$}
\put(9,0){\line(-1,-1){1}}
\put(9,0){\line(1,-3){1}}
\put(9,0){\line(0,-1){2}}
\put(8,-1){\line(1, -1){1}}
\put(9,-2){\line(1,-1){1}}
\put(8,-1){\line(1, -2){0.5}}
\put(8.5,-2){\line(1, -1){0.5}}
\put(9,-2.5){\line(2, -1){1}}

\put(8.9,0.15){$\tilde{3}$}
\put(8.95,-0.04){$\centerdot$}
\put(7.75,-1){$\tilde{2}$}
\put(7.95,-1.05){$\centerdot$}
\put(9.1,-2){$\tilde{1}$}
\put(8.95,-2.05){$\centerdot$}
\put(10.08,-3){$\tilde{0}$}
\put(9.94,-3.02){$\centerdot$}
\

\end{picture}\\
\\
\\
\\
\\
\\
\\
\\
As one can see, the resulting structure $(\tilde{M}_3, \tilde{R}_3)$ is then isomorphic to $(TC(\{3\}),\in)$. In the following we will show that this construction works in general:

\begin{lem}\label{codingx}
Let $(M_0, R)$ be a coding pair. Then the quotient structure $(\tilde{M}_0, \tilde{R})$ as defined in Definition \ref{quotientstructure} is extensional and well-founded.
\end{lem}
\begin{proof}
By Class-Comprehension $\tilde{R}\in\mathcal{C}$ and $\tilde{R}$ is well-founded as we can always find an $R$-minimal element $a$, build the equivalence class $[a]$ and find its representative $\tilde{a}$. Then $\tilde{a}$ is $\tilde{R}$ minimal as otherwise there exists $\tilde{a}'$ such that $\tilde{a}' \tilde{R} \tilde{a}$ and therefore there is $a'_0\in[a']$ such that $a'_0 R a$. 

To show that $\tilde{R}$ is extensional, let $\tilde{y}, \tilde{z}\in\tilde{M}_0$ with $\tilde{y}\neq\tilde{z}$ and assume that they have the same extension $\{\tilde{x}\,\vert\,\tilde{x}\tilde{R}\tilde{y}\}=\{\tilde{x}\,\vert\,\tilde{x}\tilde{R}\tilde{z}\}$. Going back to $(M_0, R)$ this means that the elements of the related equivalence classes $[y]$, $[z]$ have the same isomorphism types of children, i.e. for every $x_0, y_0, z_0\in M_0$ with $x_0 R y_0$, $y_0\in[y]$ and $z_0\in[z]$ we can find $x_1$ with $x_1 R z_0$ such that $x_0, x_1\in[x]$. By using property $b)$ of Definition \ref{codingdefi} it follows that the $[y]=[z]$, because we do not have multiplicities in $(M_0, R)$, i.e. isomorphic subtrees that are connected to the same $R$-predecessor. It follows that $\tilde{y}=\tilde{z}$.
\end{proof}

Note that the quotient structure always has a fixed top node which is the representative of the equivalence class of the distinguished node of $(M_0, R)$, which has the distinguished node as its only element.

It follows from Mostowski's Theorem that there is a unique transitive structure with the $\in$-relation that is isomorphic to $(\tilde{M}_0, \tilde{R})$. This structure then has the form $(TC(\{x\}),\in)$ for a unique set $x$. 

\begin{defi}
A coding pair $(M_x, R_x)$ is called a coding pair for $x$, if $x$ is the unique set such that $(\tilde{M}_x, \tilde{R}_x)$ is isomorphic to $(TC(\{x\}),\in)$. 
\end{defi}

In the following we will use this coding to associate a transitive model of SetMK$^*$ to each $\beta$-model of MK$^*$ and vice versa.

\begin{theo}\label{maintheorem}
Let $\mathcal{M}=(M,\mathcal{C})$ be a $\beta$-model of MK$^*$ and
\begin{align*}
M^+=\{x\,\vert\,\text{there is a coding pair }(M_x,R_{x})\text{ for }x\}
\end{align*}
Then $M^+$ is the unique, transitive set that obeys the following properties:
\begin{itemize}
\item[a)] $M^+\models\text{ SetMK}^*$,
\item[b)] $\mathcal{C}=P(M)\cap M^+$,
\item[c)] $M=V_{\kappa}^{M^+}$, $\kappa$ is the largest cardinal in M$^+$ and strongly inaccessible in $M^+$.
\end{itemize}
\end{theo}

The coding between $\mathcal{M}$ and $M^+$ is the key to prove the theorem. So before proving this theorem we will prove two useful fact about the coding.

As we have seen there can be more than one coding pair for an $x\in M^+$. Of course these coding pairs are isomorphic because they are all built according to Definition \ref{codingdefi} but we also would like to know that they are isomorphic in $\mathcal{M}$. For elements of $M^+$ that can be coded by sets in $\mathcal{M}$ this is trivial but for elements that are coded by proper classes we have to show the following:

\begin{lem}[Coding Lemma 1]\label{codinglem1}
Let $\mathcal{M}=(M,\mathcal{C})$ be a transitive $\beta$-model of $MK^*$. Let $N_1, N_2\in \mathcal{C}$ and $R_1, R_2$ be well-founded binary relations in $\mathcal{C}$ such that $(N_1, R_1)$ and $(N_2, R_2)$ are coding pairs as described in Definition \ref{codingdefi}. Then if there is an isomorphism between $(N_1, R_1)$ and $(N_2, R_2)$  there is such an isomorphism in $\mathcal{C}$.
\end{lem}
\begin{proof}
Let $T_1, T_2$ be the coding trees associated to the coding pairs $(N_1, R_1)$, $(N_2, R_2)$. Assume to the contrary that there is an isomorphism between $T_1$ and $T_2$ but not one in $\mathcal{C}$. It follows that the tree below the top node of $T_1$ is isomorphic to the tree below the top node of $T_2$, but there is no such isomorphism in $\mathcal{C}$. Then, as $T_1$ and $T_2$ are well-founded we can choose a $T_1$-minimal node $a_1$ of $T_1$ such that for some node $a_2$ of $T_2$ the tree $U_1$ (the tree $T_1$ below and including $a_1$) is isomorphic to $U_2$ (the tree $T_2$ below and including $a_2$) but there is no isomorphism in $\mathcal{C}$. Because of the minimality of $a_1$ we know that for every node $a_{1,i}$ of $U_1$ just below $a_1$ and every node $a_{2,j}$ of $U_2$ just below $a_2$, if $U_{1,i}$ is isomorphic to $U_{2,j}$ then there is an isomorphism in $\mathcal{C}$. Moreover the property ``$U_{1,i},U_{2,j}$ are ismorophic'' is expressible in $(M,\mathcal{C})$.\\
\\

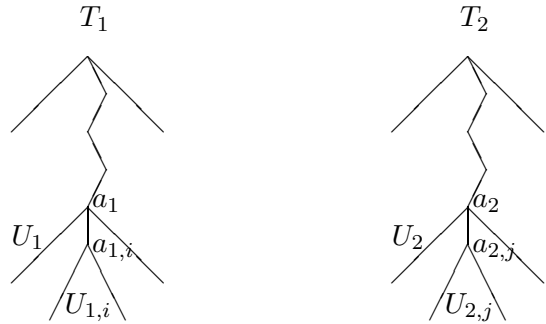
\setlength{\unitlength}{1cm} 
\begin{picture}(1,1)
\put(2.9,1.4){$T_1$}

\put(3.25,0.5){\line(-1,-2){0.25}}
\put(3,1){\line(1,-2){0.25}}
\put(3.25,-0.5){\line(-1,-2){0.25}}
\put(3,0){\line(1,-2){0.25}}
\put(3,1){\line(-1,-1){1}}
\put(3,1){\line(1,-1){1}}

\put(3,-1){\line(-1,-1){1}}
\put(3,-1){\line(1,-1){1}}
\put(3,-1){\line(0,-1){0.5}}

\put(3,-1.5){\line(-1,-2){0.5}}
\put(3,-1.5){\line(1,-2){0.5}}

\put(3.06,-1){$a_1$}
\put(2,-1,5){$U_1$}
\put(3.06,-1,6){$a_{1,i}$}
\put(2.7,-2.4){$U_{1,i}$}

\put(7.9,1.4){$T_2$}

\put(8.25,0.5){\line(-1,-2){0.25}}
\put(8,1){\line(1,-2){0.25}}
\put(8.25,-0.5){\line(-1,-2){0.25}}
\put(8,0){\line(1,-2){0.25}}
\put(8,1){\line(-1,-1){1}}
\put(8,1){\line(1,-1){1}}

\put(8,-1){\line(-1,-1){1}}
\put(8,-1){\line(1,-1){1}}
\put(8,-1){\line(0,-1){0.5}}

\put(8,-1.5){\line(-1,-2){0.5}}
\put(8,-1.5){\line(1,-2){0.5}}

\put(8.06,-1){$a_2$}
\put(7,-1,5){$U_2$}
\put(8.06,-1,6){$a_{2,j}$}
\put(7.7,-2.4){$U_{2,j}$}
\end{picture}\\
\\

Now we can apply the Class Bounding Principle of MK$^*$ to get a class $B$ so that for each $a_{1,i}$, $a_{2,j}$ for which $U_{1,i}$, $U_{2,j}$ are isomorphic, $(B)_c$ is such an isomorphism for some set $c$. Using the global well-order of $M$ we can choose a unique $c(a_{1,i}, a_{2,j})$ for each relevant pair $\langle a_{1,i}, a_{2,j}\rangle$ and combine the isomorphisms $(B)_{c(a_{1,i}, a_{2,j})}$ to get an isomorphism between $U_1$ and $U_2$ in $\mathcal{C}$, which is a contradiction.

\end{proof}
So all coding trees of the same element of $M^+$ are isomorphic in $\mathcal{C}$. For the converse it is obvious that two isomorphic coding trees code the same element in $M^+$ as they give rise to the same $(\tilde{M}_x, \tilde{R}_x)$.

The next lemma shows that we are able to see something of the coding in $M^+$:

\begin{lem}[Coding Lemma 2]\label{codinglem2}
For all $x\in M^+$ there is a one-to-one function $f\in M^+$ such that $f:\, x\to M_x$, where $(M_x, R_x)$ is a coding pair for $x$.

\end{lem}
\begin{proof}
Let $T_x$ be a coding tree for $x$ and for each $y\in x$ let $T_y$ is the subtree of $T_x$ with top node $a_y$ lying just below the top node of $T_x$ such that $T_y$ codes $y$. Note that the choice of $a_y$ is unique after having fixed the tree $T_x$.

To show that $f=\{\langle y,a_y\rangle\,|\, y\in x\}$ belongs to M$^+$, we have to find a coding tree for $f$. Firstly we construct a coding tree $T_{\langle y, a_y\rangle}$ for every $\langle y, a_y\rangle$ with $y\in x$. As $a_y$ is a set in $M$, it is a set in $M^+$ and therefore coded by some $T_{a_y}$. So we can build $T_{\langle y, a_y\rangle}$ by connecting the trees $T_y$ and $T_{a_y}$. To make sure that the relation $R_{\langle y, a_y\rangle}$ on the new tree is well-defined we can relabel the nodes of the tree $T_{a_y}$ and so we get the following picture:\\
\\

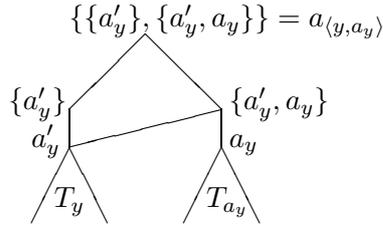
\setlength{\unitlength}{1cm} 
\begin{picture}(1,1)
\put(1,1.1){$\{\{a'_y\},\{a'_y, a_y\}\}= a_{\langle y, a_y\rangle}$}
\put(2,1){\line(-1,-1){1}}
\put(2,1){\line(1,-1){1}}

\put(3.1,-0.5){$a_y$}
\put(2.8,-1.3){$T_{a_y}$}
\put(3,-0.5){\line(-1,-2){0.5}}
\put(3,-0.5){\line(1,-2){0.5}}

\put(0.2,0){$\{a'_y\}$}
\put(3.1,0){$\{a'_y, a_y\}$}
\put(1,0){\line(0,-1){0.5}}
\put(3,0){\line(0,-1){0.5}}
\put(3,0){\line(-4,-1){2}}

\put(0.5,-0.5){$a'_y$}
\put(0.8,-1.3){$T_y$}
\put(1,-0.5){\line(-1,-2){0.5}}
\put(1,-0.5){\line(1,-2){0.5}}

\end{picture}\\
\\
\\
\\
\\
In this way we code every pair $\langle y, a_y\rangle$ with $y\in x$ and we can now join all the codes to code $f$.

Let $(M_f, R_f)$ be the following pair: $M_f=\bigcup_{z\in x} M_{\langle z, a_z\rangle}\cup\{a_f\}$ where $a_f\in M$ and $a_f\notin M_{\langle z, a_z\rangle}$ for every $z\in x$. 
Then $R_f$ is the binary relation which is defined using $R_x$ as parameter:
\begin{align*}
R_f=\{\langle v,w\rangle\,|&\,\text{ for some }y\in x\text{ either }\langle v,w\rangle\in R_{\langle y, a_y\rangle}\text{ or}\\
& v=a_{\langle y, a_y\rangle}\text{ and }w=a_f\}
\end{align*}
$M_f$ and $R_f$ are well-defined because of Class Comprehension in $MK^*$ and so $f$ is coded by the tree $T_f$ which is ordered by $R_z$ below every $a_z$ and by putting $a_{\langle y,a_y\rangle}$ below $a_f$ otherwise.

\end{proof}

Now we give the proof of Theorem \ref{maintheorem}.

\begin{proof}
a) We show that if $\mathcal{M}$ is a $\beta$-model of $MK^*$ then $M^+\models SetMK^*$. The first step is proving that $M^+$ satisfies $ZFC^-$ with Set-Bounding. 

Observe that $M^+$ is transitive: Let $x\in M^+$. Then for every $y\in x$ there is a coding tree for $y$ (namely the corresponding subtree of $T_x$). Therefore $y\in M^+$ and so $x\subseteq M^+$. From transitivity it follows that Extensionality and Foundation hold in $M^+$; Infinity follows as $\omega\in M^+$.

Pairing: Let $x,y$ be coded by $T_x, T_y$ respectively. Then $\{x,y\}$ is coded by the tree:\\
\\

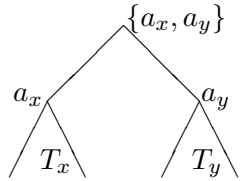
\setlength{\unitlength}{1cm} 
\begin{picture}(1,1)

\put(3,1){\line(-1,-1){1}}
\put(3,1){\line(1,-1){1}}

\put(2,0){\line(-1,-2){0.5}}
\put(2,0){\line(1,-2){0.5}}

\put(4,0){\line(-1,-2){0.5}}
\put(4,0){\line(1,-2){0.5}}

\put(3.03,1){$\{a_{x},a_y\}$}
\put(4.02,0){$a_y$}
\put(3.9,-0.9){$T_y$}
\put(1.55,0){$a_x$}
\put(1.9,-0.9){$T_x$}

\end{picture}\\
\\
\\
\\
Union: Let $x$ be coded by $T_x$:\\
\\

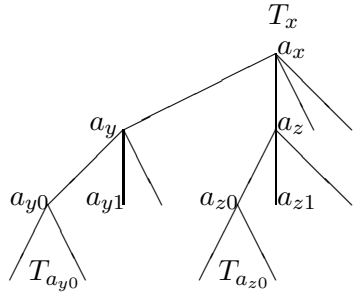
\setlength{\unitlength}{1cm} 
\begin{picture}(1,1)
\put(5.9,1.4){$T_x$}

\put(6,1){\line(0,-1){1}}
\put(6,1){\line(-2,-1){2}}
\put(6,1){\line(1,-1){1}}
\put(6,1){\line(1,-2){0.5}}

\put(6,0){\line(0,-1){1}}
\put(6,0){\line(1,-1){1}}
\put(6,0){\line(-1,-2){0.5}}
\put(5.5,-1){\line(-1,-2){0.5}}
\put(5.5,-1){\line(1,-2){0.5}}

\put(4,0){\line(0,-1){1}}
\put(4,0){\line(-1,-1){1}}
\put(4,0){\line(1,-2){0.5}}
\put(3,-1){\line(-1,-2){0.5}}
\put(3,-1){\line(1,-2){0.5}}

\put(6.02,1){$a_{x}$}
\put(6.02,0){$a_z$}
\put(5,-1){$a_{z0}$}
\put(5.25,-2){$T_{a_{z0}}$}
\put(6.02,-1){$a_{z1}$}

\put(3.55,0){$a_y$}
\put(3.5,-1){$a_{y1}$}
\put(2.5,-1){$a_{y0}$}
\put(2.75,-2){$T_{a_{y0}}$}

\end{picture}\\
\\
\\
\\
\\
\\
The obvious way to code $\bigcup x$ would be to join the $a_{y0}, a_{y1}, \ldots, a_{z0}, a_{z1}, \ldots$ together by one top node $a_{\bigcup x}$. But in general this is not a coding tree by reasons of isomorphism: Our coding trees have the property that subtrees which are connected to the same node on the next level above are all pairwise non-isomorphic. In this case that means that the trees $T_{a_y}, T_{a_z}, \ldots$ are pairwise non-isomorphic, as are the trees $T_{a_{y0}}, T_{a_{y1}}, \ldots$ and the trees $T_{a_{z0}}, T_{a_{z1}}, \ldots$ and so on. But, as we explained before, it can happen that some of the $T_{a_{yi}}$ are isomorphic to, for example, some of the $T_{a_{zj}}$. So if we connect these trees by a top node the resulting tree would have isomorphic subtrees connected by the same node on the next level and therefore would not be a coding tree. This problem can easily be resolved by taking equivalence classes of the subtrees of $T_x$ from the second level below $a_x$ (where two trees are equivalent if the are isomorphic). Then we take a representative from each equivalence class and connect them to the top node $a_{\bigcup x}$ (as before, this is possible by Class Comprehension in $MK^*$ and Coding Lemma 1). 

To prove Comprehension and Bounding we need to take a closer look at how formulas in $M^+$ translate to formulas in $\mathcal{M}$:
\begin{lem}\label{codingformula}
For each first-order formula $\varphi$ there is a formula $\psi$ of second-order class theory such that for all $x_1, \ldots, x_n\in M^+$, $M^+\models\varphi(x_1,\ldots,x_n)$ if and only if $\mathcal{M}\models\psi(c_1, \ldots, c_n)$ for any choice of codes $c_1, \ldots, c_n$ for $x_1,\ldots, x_n$.
\end{lem}
\begin{proof}
The proof is by induction over the complexity of the formula $\varphi$. For the first atomic case assume that $M^+\models y\in x$. Let $c_x$ and $c_y$ be codes for $x$ and $y$ respectively and let $T_x, T_y$ be the associated coding trees. As we know that $y\in x$ it follows that there is a direct subtree $T_{y'}$ of $T_x$ such that $T_{y'}$ is a coding tree for $y$ ( ``direct subtree'' means a subtree whose top node lies just below the top node of the original tree). As $T_{y'}$ and $T_y$ are both codes for $y$ they are isomorphic and by Coding Lemma $1$ we know that they are isomorphic in $\mathcal{M}$. So $\mathcal{M}\models\text{``$c_y$ is isomorphic to a direct subtree of $c_x$''}$ and this therefore is the desired $\psi$.

For the second atomic case assume that $M^+\models y=x$. Let $c_x$ and $c_y$ be codes for $x$ and $y$ respectively. As $y=x$, $c_y$ is also a code for $x$ and again by Coding Lemma $1$ we know that the codes are isomorphic in $\mathcal{M}$ thus giving us the desired $\psi$.

The cases of $\neg\varphi$, $\varphi_1\wedge\varphi_2$  follow easily by using the induction hypothesis. For the quantor case assume that $M^+\models\forall x\varphi$. Let $c_x$ be a code for $x$. By induction hypothesis let $\psi$ be the second-order formula associated to $\varphi$ such that $\mathcal{M}\models\psi$. Then $M^+\models\forall x \varphi$ translates to $\mathcal{M}\models\forall c_x \psi$.
\end{proof}

Comprehension: Let $a, x_1, \ldots, x_n\in M^+$ and let $\varphi(x,x_1, \ldots, x_n,a)$ be any first-order formula. We will show that $b=\{x\in a: M^+\models\varphi(x,x_1, \ldots, x_n,a)\}$ in an element of M$^+$ by using Class Comprehension in $\mathcal{M}$ to find the corresponding $B\in \mathcal{C}$ and build from it a coding tree for $b$.

Let $T_{x_1}, \ldots, T_{x_n},T_a$ be codes for the corresponding elements of $M^+$ and let $\psi$ be the formula corresponding to $\varphi$ provided by Lemma \ref{codingformula}. Assume that $b$ is non-empty, i.e. that there is  $x_0$ in $a$ such that $\varphi$ holds. Therefore there is a $c_0$ such that $\psi(c_0, T_{x_1}, \ldots, T_{x_n},T_a)$ holds. Let $c$ be a variable that varies over the level directly below the top level of $T_a$ so that each $T_{a(c)}$ denotes a direct subtree of $T_a$. Then by Class Comprehension there is a class $B$ such that if $\psi(T_{a(c)}, T_{x_1}, \ldots, T_{x_n},T_a)$ holds then $(B)_c$ is the direct subtree $T_{a(c)}$ of $T_a$ and if not then $(B)_c$ is $T_{c_0}$. 

So let $T_b$ be the coding tree with top node $a_b$ and whose direct subtrees are all of the $(B)_c$:
\\
\\

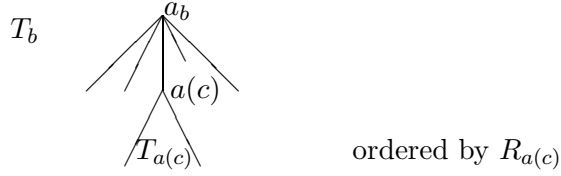
\setlength{\unitlength}{1cm} 
\begin{picture}(1,1)
\put(1,0.7){$T_b$}

\put(3,1){\line(0,-1){1}}
\put(3,1){\line(-1,-1){1}}
\put(3,1){\line(1,-1){1}}
\put(3,1){\line(-1,-2){0.5}}
\put(3,1){\line(1,-2){0.3}}
\put(3,0){\line(-1,-2){0.5}}
\put(3,0){\line(1,-2){0.5}}
\put(3.02,1){$a_b$}
\put(3.1,-0.1){$a(c)$}
\put(2.65,-0.9){$T_{a(c)}$}

\put(5.5,-0.9){ordered by $R_{a(c)}$}

\end{picture}\\
\\
\\
\\
Then $T_b$ codes $b\in M^+$ with $b=\{x\in a: \varphi(x, x_1,\ldots, x_n,a)\}$.\\

Bounding: We have to show that for $a\in M^+$ and $\varphi$ a first-order formula
\begin{equation*}
M^+\models\forall x\in a\,\exists y\,\varphi(x,y)\to\exists b\,\forall x\in a\,\exists y\in b\,\varphi(x,y).
\end{equation*}
So assume that $\forall x\in a\,\exists b\,\varphi(x,y)$.
Let $T_y, T_a$ be coding trees for $y$ and $a$ respectively and let $\psi$ be the second-order formula corresponding to $\varphi$ provided by Lemma \ref{codingformula}. By Class-Bounding in MK$^*$ we know that
\begin{equation*}
\exists B\,\forall T_x\text{ direct subtree of } T_a\exists y'\,\psi(T_x, (B)_{y'}),
\end{equation*}
where $(B)_{y'}=\{z\,|\,(y',z)\in B\}$. By Class Comprehension we can join together all the section $(B)_{y'}$ which are coding trees $T_{(B)_{y'}}$ to obtain a tree $T_b$ with top node $a_b$ such that the $T_{(B)_{y'}}$ are the direct subtrees of $T_b$. It follows that in $\mathcal{M}$ there is a tree $T_b$ such that for every tree $T_x$ subtree of $T_a$ there is a $T_{(B)_{y'}}$ direct subtree of $T_b$ such that $\psi(T_x, T_{(B)_{y'}})$ and the tree $T_b$ gives us the desired $b$ in $M^+$.\\

Replacement: Follows from Comprehension and Bounding.\\

Choice: We have to show that every element of $M^+$ can be well-ordered (we aim for the strongest version of the axiom of Choice in a set-theory without Power Set (see \cite{Zarach:1982}). So let $x\in M^+$ and let $T_x$ be a coding tree for $x$ with top node $a_x$. We know that the direct subtrees $T_y$ of $T_x$ code the elements $y$ of $x$ and their top nodes $a_y$ are elements of $M$. As we have a well-order of $M$ we can well-order the class $B=\{a_y\,|\, a_y$ is the top node of a direct subtree $T_y$ of $T_x\}$. We call this well-order $W$. Now we can build a tree for every pair $\langle a_y, a_z\rangle\in W$ by using the trees $T_y, T_z$ analogous as we did in the proof of Coding Lemma $2$:\\
\\

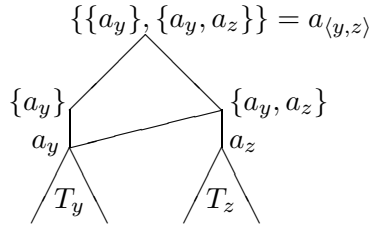
\setlength{\unitlength}{1cm} 
\begin{picture}(1,1)
\put(1,1.1){$\{\{a_y\},\{a_y, a_z\}\}= a_{\langle y, z\rangle}$}
\put(2,1){\line(-1,-1){1}}
\put(2,1){\line(1,-1){1}}

\put(3.1,-0.5){$a_z$}
\put(2.8,-1.3){$T_z$}
\put(3,-0.5){\line(-1,-2){0.5}}
\put(3,-0.5){\line(1,-2){0.5}}

\put(0.2,0){$\{a_y\}$}
\put(3.1,0){$\{a_y, a_z\}$}
\put(1,0){\line(0,-1){0.5}}
\put(3,0){\line(0,-1){0.5}}
\put(3,0){\line(-4,-1){2}}

\put(0.5,-0.5){$a_y$}
\put(0.8,-1.3){$T_y$}
\put(1,-0.5){\line(-1,-2){0.5}}
\put(1,-0.5){\line(1,-2){0.5}}

\end{picture}\\
\\
\\
\\
\\
So for every $\langle a_y, a_z\rangle\in W$ we get a coding tree for the pair $\langle y,z\rangle$ with $y,z\in x$. As we have shown in the proof of Coding Lemma 2 we can now join together the trees  by a single top node $a_w$ using Class Comprehension. We now get a tree $T_w$ which is a coding tree for an element $w$ of $M^+$ and $w$ is a well-order of $x$.\\
\begin{rem}
The next two results below (b and c) will show, that there even is a global choice function for the sets in $V_{\kappa}^{M^+}$ for $\kappa$ an inaccessible cardinal, as there is a class which well-orders $M$ and we will show that every class in $\mathcal{C}$ is an element of $M^+$.
\end{rem}

b) We have to show that $\mathcal{C}=P(M)\cap M^+$. So assume that $X\in \mathcal{C}$ and $y\in X$. Then $y\in M$ and so can be coded by the following tree: $y$ is the top node of the tree $T_y$. On the first level below the top node there are nodes for every element of $y$ which are named by pairwise different elements $z_i$ of $M\setminus\{y\}$. On the first level below such an $z_i$ there are nodes for every element in $z_i$ named by pairwise different elements $v_j$ of $M\setminus\{y, z_i\}$ and so on. So $T_y$ is a coding tree for $y$ and  therefore $y\in M^+$. This can be done for all $y\in X$ and by Class Comprehension the trees $T_y$ can be connected to a tree $T_X$ with top node $a_X$. Then the pair $(M_X, R_X)$ gives a code for $X$ with $M_X= \bigcup_{y\in X} M_y\cup\{a_X\}$ and 
\begin{equation*}
R_X=\{\langle v,w\rangle\,|\text{ for some }y\in X\text{ either }\langle v,w\rangle\in R_y\text{ or }v=a_y\text{ and }w=a_X\}
\end{equation*}
Therefore $X\in M^+$. 

For the converse, let $x\in M^+$ and $x\subseteq M$. Then there exists a coding pair $(M_x, R_x)$ of $x$ such that $(\tilde{M}_x, \tilde{R}_x)\cong(TC\{x\}, \in)$ (see Lemma \ref{codingx}). As $(\tilde{M}_x, \tilde{R}_x)$ is in $\mathcal{C}$, has rank Ord($M$) and we can build $TC(\{x\})$ by transfinite induction from $(\tilde{M}_x, \tilde{R}_x)$, we can decode $x$ in $\mathcal{C}$ and so $x\in\mathcal{C}$.\\

c) Now we will show that there is a strongly inaccessible cardinal $\kappa$ in $M^+$ which is the largest cardinal in $M^+$ and the elements of $M$ (the sets in $\mathcal{M}$) are exactly the elements of $V_{\kappa}^{M^{+}}$. 

Let $\kappa$ be $Ord(M)$. Then as $\kappa\subseteq M$ and $\kappa\in M^+$ it follows from b) that $\kappa$ is a class in $\mathcal{C}$. Let $f: \beta \to\kappa$ with $\beta$ is a ordinal less than $\kappa$ be a function in $\mathcal{C}$. From the Class Bounding Principle it follows that $f$ is bounded in $\kappa$. So $\kappa$ is regular in $\mathcal{M}$ and therefore regular in M$^+$. Moreover, again by b), any subset of an ordinal $\beta$ of $M$ which belongs to $M^+$ is a class in $\mathcal{C}$ and indeed a set in $M$, so the power set of $\beta$ in $M^+$ equals the power set of $\beta$ in $M$ and so $\kappa$ is strongly inaccessible. It follows that if $x\in M$ then $x\in V_{\kappa}^{M^{+}}$. For the converse let $x\in V_{\kappa}^{M^{+}}$ and let $(M_x, R_x)$ be a coding pair and $T_x$ the associate coding tree for $x$ . By Coding Lemma 2 any coding tree of a set is a set, so $T_x$ is an element of $M$. Clause 3 of the axioms of SetMk$^*$ follows directly from Coding Lemma 2 and so $\kappa$ is the largest cardinal in M$^+$.\\ 

That $M^+$ is unique follows from its construction: Let $M^{++}$ be another such model of SetMK$^*$ (i.e. it is transitive, $\mathcal{C}=P(M)\cap M^{++}$ and $M=V_{\kappa}^{M^{++}}$ with $\kappa$ largest cardinal in $M^{++}$ and strongly inaccessible cardinal in $M^{++}$). Then M$^+$ and M$^{++}$ have the same largest cardinal $\kappa$, they have the same subsets of $\kappa$ and as every set in both models can be coded by a subset of $\kappa$ they are the same.

This concludes the proof of Theorem \ref{maintheorem}.
\end{proof}

The converse of Theorem \ref{maintheorem} follows by the corresponding axioms in the SetMK$^*$ model:
\begin{theo}\label{conversemain}
Let $N$ be a transitive model of SetMK$^*$ that has a strongly inaccessible cardinal $\kappa$ that is the largest cardinal, let $\mathcal{C}= P(M)\cap N$ and $M$ is defined to be $V^{N}_\kappa$. Then $\mathcal{M}=(M, \mathcal{C})$ is a $\beta$-model of MK$^*$ and the model $M^+$ derived from $\mathcal{M}$ by Theorem \ref{maintheorem} equals $N$.
\end{theo}
\begin{proof}
We have to show that $(M, \mathcal{C})$ fulfills the axioms of MK$^*$: Extensionality, Pairing, Infinity, Union, Power Set, and Foundation follow directly by the corresponding axioms of SetMK$^*$. By the definition of $M$ and $\mathcal{C}$ it follows that every set is a class and elements of classes are sets. 

For the remaining axioms, note that there is an easy converse for Lemma \ref{codingformula}: For each formula $\varphi$ of second-order class theory there is a first-order formula $\psi$ such that for all $x_1, \ldots, x_n\in\mathcal{M}$, $\mathcal{M}\models\varphi(x_1, \ldots, x_n)$ if and only if $N\models\psi(x_1, \ldots, x_n)$. This holds because by assumption all elements of $\mathcal{M}$ are elements of $\mathcal{C}$ or $M$ and therefore elements of $N$ and so $\varphi$ and $\psi$ are the same where the statement that $x$ is a set in $\mathcal{M}$ translates to $x\in V^{N}_\kappa$ and the statement that $X$ is a class in $\mathcal{M}$ translates to $X\in  P(M)\cap N$. So for Class Comprehension we have to show that the following holds:
\begin{equation*}
\forall X_1\ldots\forall X_n\exists Y\; Y=\{x: \varphi(x,X_1,\ldots, X_n)\}
\end{equation*}
where $\varphi$ is a formula containing class parameters in which quantification over both sets and classes is allowed. By the definition of $M$ and $\mathcal{C}$ this statement is exactly the Comprehension Axiom of $N$ where $\psi$ is the first-order formula corresponding to $\varphi$: $y=\{x\in V^{N}_\kappa: N\models\psi(x, x_1, \ldots, x_n, V^{N}_\kappa)$. 

For Class Bounding we have to show:
\begin{equation*}
\forall x\,\exists A\, \varphi(x,A)\to\exists B\,\forall x\,\exists y\, \varphi(x,(B)_{y})
\end{equation*}
where $(B)_{y}=\{z\,\vert\, (y,z)\in B\}$. So assume that $\forall x\,\exists A\, \varphi(x,A)$ holds in $\mathcal{M}$. Then translating this to $N$ we know by Set-Bounding that
\begin{equation*}
\forall x\in V_\kappa^{M^+} \,\exists A\in P(M)\cap M^+  \,\psi(x,A)\to \exists b\,\forall x\in V_\kappa^{M^+}\,\exists y\in b\,\psi(x,y)
\end{equation*}
where $\psi$ is the first-order formula corresponding to $\varphi$. By Set-Comprehension we can form a set $b_0$ from $b$ such that $b_0=\{y\,\vert\, y\in b\wedge y\subseteq V_{\kappa}^{N}\}$. Then there is a function $f\in N$ from $V_{\kappa}^{N}$ onto $b_0$ (as $b_0$ has size less or equal $\kappa$) and so $f$ is also an element of $\mathcal{M}$. The we can define the class $(B)_z=\{w\,\vert\, w\in f(z)\}$ and therefore also $B=\{(z,w)\,\vert\, z\in V_{\kappa}^{N}\wedge w\in f(z)\}$. So Class-Bounding holds.

For Global Choice we have to show that there is a well-ordering of $M$. We know that every element of $N$ can be well-ordered and so $V^{N}_\kappa$ can be well-ordered. The well-order is therefore an element of $\mathcal{C}$.

$(M, \mathcal{C})$ has to be a $\beta$-model: Any well-founded relation in $(M,\mathcal{C})$ corresponds to a well-founded relation in $N$ and because $N$ is a transitive model of ZF$^-$, well-foundedness is absolute (we can define a rank function into the ``real'' ordinals which witnesses the well-foundedness in $V$).

Finally when we build the $M^+$ of $\mathcal{M}$ according to Theorem \ref{maintheorem}, $M^+$ and $N$ are both transitive, have the same largest cardinal $\kappa$ and the same subsets of $\kappa$ and are therefore equal.
\end{proof}

\begin{rem}
We can also use this switching between models of MK$^*$ and SetMK$^*$ for class-forcing: Instead of doing class-forcing over MK$^*$ we go to SetMK$^*$ and do a set-forcing there. Note that by doing this indirect version of class-forcing we don't lose the tameness requirement for the forcing: Assume the class-forcing is not tame (as for example a forcing which collapses the universe to $\omega$). Then we go to $M^+\models SetMK^*$ and force with the associated set-forcing. But such a forcing destroys the inaccessibility of $\kappa$ and therefore the preservation of PowerSet in the MK$^*$ extension $\mathcal{M}[G]$.
\end{rem}

\begin{cor} \label{L}
\begin{equation*}
M^+=\bigcup_{C\in\mathcal{C}} L_{\kappa^*}(C).
\end{equation*}
where $\kappa^*$ is the height of $M^+$ and 
\begin{align*}
L_0 (C) = & \,TC(\{C\})\\
L_{\beta + 1} (C) = & \,\text{Def}\,(L_{\beta} (C))\\
L_{\lambda} (C) = & \bigcup_{\beta<\lambda} L_{\beta} (C), \,\lambda \text{ limit}.
\end{align*}
\end{cor}
\begin{proof}
Let $x\in M^+$. Then there is a coding pair $(M_x, R_x)$ for $x$ such that $(\tilde{M}_x, \tilde{R}_x)$ is isomorphic to $(TC(\{x\}), \in)$. As $\tilde{M}_x$ and $\tilde{R}_x$ are elements of $\mathcal{C}$ we can code the pair $(\tilde{M}_x, \tilde{R}_x)$ by a class $C_x\in \mathcal{C}$. As $C_x$ is an element of M$^+$, $L_{\kappa^*}(C_x)$ is an inner model in M$^+$. But now we can decode $x$ in $L_{\kappa^*}(C_x)$ as we can build $(TC(\{x\})$ by transfinite induction from $(\tilde{M}_x, \tilde{R}_x)$. So $x\in L_{\kappa^*}(C_x)$.

For the converse, let $x\in\bigcup_{C\in\mathcal{C}} L_{\kappa^*}(C)$, i.e. there is an $C_x\in\mathcal{C}$ such that $x\in L_{\kappa^*}(C_x)$. As $L_{\kappa^*}(C_x)$ is an inner model of $\mathcal{M}$, $x$ is an element of $\mathcal{C}$ and by Theorem \ref{maintheorem} b) it is an element of M$^+$.
\end{proof}

\section{Hyperclass Forcing and Forcing in SetMK$^{**}$}
In the last section we have seen how to move back and forth between a model of MK$^*$ and its associated SetMK$^*$ model. Now we will use this relation between a model of class theory and a model of set theory to define hyperclass forcing. A hyperclass is a collection whose elements are classes. The key idea is that instead of trying to formalize forcing for a definable hyperclass forcing notion, we can go to the associated model of SetMK$^*$ where the forcing notion is now a class and so we force with a definable class forcing there and then go back to a new MK$^*$ model. First let us define the relevant notions:
\begin{defi}
Let $\mathcal{M}=(M, \mathcal{C})$ be a model of MK$^*$ and for $\mathbb{P}\subseteq \mathcal{C}$ let $(\mathbb{P}, \leq)=\mathbb{P}$ be an $\mathcal{M}$-definable partial ordering with a greatest element $1^{\mathbb{P}}$. $P, Q\in \mathbb{P}$ are compatible if for some $R$, $R\leq P$ and $R\leq Q$. A definable hyperclass $D\subseteq \mathbb{P}$ is dense if $\forall P\exists Q (Q\leq P\text{ and } Q\in D)$. Then a $\mathbb{G}\subseteq\mathcal{C}$ is called a $\mathbb{P}$-generic hyperclass  over $\mathcal{M}$ iff $\mathbb{G}$ is a pairwise compatible, upward-closed subcollection of $\mathbb{P}$ which meets every dense subcollection of $\mathbb{P}$ which is definable over $\mathcal{M}$.
\end{defi}

We will assume that for each $P\in \mathbb{P}$ there exists $\mathbb{G}$ such that $P\in\mathbb{G}$ and $\mathbb{G}$ is $\mathbb{P}$-generic over $\mathcal{M}$ (this is always possible if the model $\mathcal{M}$ is countable).

To define the structure $(M, \mathcal{C})[\mathbb{G}]$ where $\mathbb{G}$ is a $\mathbb{P}$-generic hyperclass over $(M, \mathcal{C})$ we will use Theorem \ref{maintheorem} and Proposition \ref{conversemain}. By Theorem \ref{maintheorem} we go to the model M$^+\models SetMK^*$. As $\mathbb{P}$ is a subcollection of $\mathcal{C}$ in $\mathcal{M}$ it becomes a subclass of $P(M)\,\cap\, M^+$ and is an $M^+$-definable class, $\mathbb{G}$ remains a pairwise compatible, upward-closed subclass of $\mathbb{P}$ which meets every dense subclass of $\mathbb{P}$ which is definable over $M^+$ and therefore is definable class-generic over $M^+$. Then we define names, their interpretation and the extension of M$^+$ as usual: A $\mathbb{P}$-name in M$^+$ is a set in M$^+$ consisting of pairs $(\tau, p)$ where $\tau$ is a $\mathbb{P}$-name in M$^+$ and $p$ belongs to $\mathbb{P}$ (as we are in the set model we now denote the elements of $\mathbb{P}$ with lower-case letters). Then $\mathcal{N}=\cup\{\mathcal{N}_{\alpha}\,\vert\, \alpha\in Ord(M^+)\}$ is the collection of all names where $\mathcal{N}_0=\emptyset$, $\mathcal{N}_{\alpha+1}=\{\sigma\,\vert\,\sigma\text{ is a subset of }\mathcal{N}\times P\text{ in }M^+\}$ and $\mathcal{N}_{\lambda}=\cup\{\mathcal{N}_{\alpha}\,\vert\,\alpha<\lambda\}$ for a limit ordinal $\lambda$. For a $\mathbb{P}$-name $\sigma$ its interpretation is $\sigma^{\mathbb{G}}=\{\tau^{\mathbb{G}}\,\vert\, p\in \mathbb{G}\text{ for some }(\tau, p)\in\sigma\}$. Then $M^+[\mathbb{G}]$ is the set of all such $\tau^{\mathbb{G}}$. Finally we can define the extension of $\mathcal{M}$:

\begin{defi}
Let $\mathcal{M}=(M, \mathcal{C})$ be a $\beta$-model of $MK^*$, $\mathbb{P}$ be a definable hyperclass forcing and $\mathbb{G}\subseteq\mathbb{P}$ be a $\mathbb{P}$-generic hyperclass over $\mathcal{M}$. Let $M^+$ be the model of SetMK$^*$ associated to $\mathcal{M}$ by Theorem \ref{maintheorem} and assume that $M^+[\mathbb{G}]\models SetMK^*$ with largest cardinal $\kappa$ with $M^+[\mathbb{G}]$ transitive. Then $\mathcal{M}[\mathbb{G}]=(M, \mathcal{C})[\mathbb{G}]$ is the $\beta$-model of MK$^*$ derived from $M^+[\mathbb{G}]$ by Theorem \ref{conversemain}, whose sets are the elements of $V_{\kappa}^{M^+[\mathbb{G}]}$ and whose classes are the subsets of $V_{\kappa}^{M^+[\mathbb{G}]}$ in $M^+[\mathbb{G}]$, where $\kappa$ is the largest cardinal of $M^+[\mathbb{G}]$ and is strongly inaccessible. Such a model is called a definable hyperclass-generic outer model of $\mathcal{M}$.
\end{defi}

This definition assumes that the definable class-forcing $\mathbb{P}$ again produces a model of SetMK$^*$ with the same largest cardinal $\kappa$ where $\kappa$ is strongly inaccessible (we say in short that $\mathbb{P}$ does not change $\kappa$). Unfortunately the assumption that SetMK$^*$ is preserved is not as straightforward as it might seem. Definable class-forcing was developed by \cite{Friedman:2000}. There the concept of pretameness and tameness of a forcing notion is introduced and it is shown that such a forcing has a definable forcing relation and preserves the axioms. In the case of SetMK$^*$ we now have the added problem that we are not forcing over a model of full ZFC but rather over ZFC$^-$, i.e. without the Power Set Axiom. This can cause problems when we use concepts like the hierarchy of the $V_{\alpha}$, for example to prove that pretame class-forcings preserve the Replacement (or in our case the Set-Bounding) Axiom. So we cannot simply transfer the results of \cite{Friedman:2000} but have to prove the Definability Lemma and the preservation of the axioms again without making use of the Power Set Axiom. 

To define definable class-forcing in SetMK$^*$ first note that the following still holds:
Let $M^+$ be a transitive model of SetMK$^*$, $P$ be a $M^+$-definable forcing notion and $G$ $P$-generic over $M^+$. Then $M^+[G]$  is transitive and $Ord(M^+[G])= Ord(M^+)$. It follows from the definition of the interpretation of names and the definition of $M^+[G]$ that if $y\in\sigma^G$ then $y=\tau^G$ for some $\tau\in TC(\sigma)$ and therefore $M^+[G]$ is transitive.
Furthermore for every $x\in Ord(M^+)$ there exists a name $\sigma$ for $x$ (i.e. $x=\sigma^G$ as defined above) with name-rank of $\sigma$ $=$ the least $\alpha\in Ord(M)$ such that $\sigma\in\mathcal{N}_{\alpha+1}$ and by induction the von Neumann rank of $\sigma^G$ is at most the name rank of $\sigma$. So we know that if ``new'' sets are added by the forcing they have size at most the ``old'' sets from $M^+$ and so $Ord(M^+[G])\subseteq Ord(M^+)$.

We will first treat the case where we already assume that the forcing relation is definable and $P$ is a  pretame class-forcing and then show how we can ensure that in general pretame class-forcings  preserve the axioms and the Definability Lemma holds.

\begin{prop}\label{forcingsetmk}
Let $M^+$ be a model of SetMK$^*$ and let $P$ be a pretame definable class-forcing over M$^+$ that does not change $\kappa$ and whose forcing relation is definable. Let $G\subseteq P$ be definable class-generic over M$^+$. Then $M^+[G]$ is a model of SetMK$^*$.
\end{prop}
\begin{proof}
Extensionality, Pairing, Comprehension, Infinity, Foundation and Choice still hold by the proof for definable class-forcing over full ZFC. We have to show that Set-Bounding holds in $M^+[G]$, i.e. 
\begin{equation*}
M^+[G]\models\forall x\in a \,\exists y \,\varphi(x,y)\to \exists b\,\forall x\in a\,\exists y\in b\,\varphi(x,y)
\end{equation*}
Let $\sigma$ be a name for $a$. We can extend any $p$ for which $p\Vdash\forall x\in\sigma\,\exists y\,\varphi(x,y)$ to force that there is an isomorphism between $\sigma$ and an ordinal $\alpha$ (by using AC) and so we can assume without loss of generality that $\sigma$ is $\check{\alpha}$ where $\alpha\in Ord$ and therefore $p\Vdash\forall x<\alpha\,\exists y\,\varphi(x,y)$. Then for such a fixed $p$ and for each $x<\alpha$ we can define by the Definability of the forcing relation $D_x=\{ q\leq p\,\vert\,\exists\tau\, q\Vdash\varphi(x,\tau)\}$ where $D_x$ is dense below $p$. By pretameness there is a $q\leq p$ and $\langle d_x\,\vert\, x<\alpha\rangle\in M^+$ such that for all $x<\alpha$, $d_x$ is pretense $\leq q$ and by genericity there is such a $q$ in $G$. Then we know that for all pairs $\langle x,r\rangle$ where $x<\alpha$ and $r\in d_x$ there is $\tau$ such that $r\Vdash \varphi(x,\tau)$. By the Set-Bounding principle in $M^+$ we get a set $T\in M^+$ such that $\forall (x,r)$ with $r\in d_x$ $\exists\tau\in T$ such that $r\Vdash\varphi(x,\tau)$. Finally let $\pi$ be a name for $\{\tau^G\,\vert\,\tau\in T\}$, i.e. $\pi=\{\langle\tau, 1^{\mathbb{P}}\rangle\,\vert\,\tau\in T\}$. Then, because the generic below $q$ hits every $d_x$, $\varphi(x,\tau)$ will hold for some $\tau\in T$. It follows that $q\Vdash\forall x<\alpha\,\exists y\in \pi\,\varphi(x,y)$.
Then Union follows with the use of Set-Bounding.
\end{proof}

With this proposition we have shown that in a model of MK$^*$ we can force with a definable  hyperclass-forcing $\mathbb{P}$ and preserve MK$^*$, provided $\mathbb{P}$ translates to a pretame class-forcing in SetMK$^*$ 
which preserves the inaccessibility of $\kappa$ and whose forcing relation is definable. But in practice we don't usually know if the forcing relation is definable, even if we know that $P$ is pretame due to the absence of a suitable hierarchy (like the $V$-hierarchy which suffices when forcing over ZF-models). So we will introduce a preparatory forcing which does not add any new sets but converts the SetMK$^*$ model $M^+$ into a model of the form  $L_{\alpha}[A]$ for some generic class predicate $A\subseteq ORD$ preserving SetMK$^*$ (relative to $A$). This will allow us to use the relativized $L$ hierarchy and therefore adapt the proof of  the Definability Lemma for a pretame class-forcing and the fact that it preserves the axioms. Such a preparatory forcing presents us with two difficulties: first we have to show that its forcing relation is definable and the forcing is pretame, so that we can infer from Proposition \ref{forcingsetmk} that it preserves the axioms.  Secondly we have to show that the predicate $A$, that was added by the forcing, can be coded into a subset of $\kappa$ so as to avoid problems when going back to the MK$^*$ model.

To prove the pretameness of such a forcing we have to add a new axiom to SetMK$^*$, namely a variant of Dependent Choice. To ensure that this axiom holds in M$^+$, we will add its class version to MK$^*$ and show that it is transformed to the appropriate set version using the coding introduced in the last section.
\begin{defi}
Let MK$^{**}$ consist of the axioms of MK$^*$ plus Dependent Choice for Classes (we denote this with $DC_{\infty}$):
\begin{equation*}
\forall \vec{X}\exists Y \varphi(\vec{X},Y)\to\forall X\exists\vec{Z}\,(Z_0=X\wedge\forall i\in ORD\,\varphi(\vec{Z}\restriction i, Z_{i}))
\end{equation*}
 where $\vec{X}$ is an $\alpha$-length sequence of classes for some $\alpha\in ORD$, $\vec{Z}$ is an ORD-length sequence of classes and $Z\restriction i$ is the sequence of the ``previously chosen'' $Z_j$, $j<i$. 
 \end{defi}
 
 In the resulting SetMK$^{**}$ model M$^+$, $DC_{\infty}$ becomes a form of $\kappa$-Dependent Choice:
\begin{equation*}
\forall \vec{x}\,\exists y \varphi(\vec{x},y)\to\forall x\exists\vec{z}\,(z_0=x\wedge\forall i<\kappa\,\varphi(\vec{z}\restriction i, z_{i}))
\end{equation*}
 where $\vec{x}$ is a $<\kappa$-length sequence of sets, $\vec{z}$ is a $\kappa$-length sequences of sets and $z\restriction i$ is the sequence of the ``previously chosen'' $z_j$, $j<i$.

The coding between MK$^{**}$ and SetMK$^{**}$ works exactly as in the MK$^*$ case, we only have to prove that it transforms $DC_{\infty}$ into $DC_{\kappa}$ and vice versa.

\begin{prop}
\begin{enumerate}
\item Let $\mathcal{M}=(M,\mathcal{C})$ be a $\beta$-model of MK$^{**}$. Then we can define a model
\begin{align*}
M^+=\{x\,\vert\,\text{there is a coding pair }(M_x,R_{x})\text{ that codes }x\}
\end{align*}
Then $M^+$ is the unique, transitive set that obeys the following properties:
\begin{itemize}
\item[a)] $M^+\models\text{ SetMK}^{**}$,
\item[b)] $\mathcal{C}=P(M)\cap M^+$,
\item[c)] $M=V_{\kappa}^{M^+}$, $\kappa$ is the largest cardinal in $M^+$ and strongly inaccessible in $M^+$.
\end{itemize}

\item Let M$^+$ be a model of SetMK$^{**}$ that has a strongly inaccessible cardinal $\kappa$, let $\mathcal{C}= P(M)\cap M^+$ and $M= V^{M^+}_\kappa$. Then $\mathcal{M}=(M, \mathcal{C})$ is a model of MK$^{**}$.
\end{enumerate}
\end{prop}

\begin{proof}

For 1.: Using the proof of Theorem \ref{maintheorem} it only remains to show that $M^+$ is a model of $\kappa$-Dependent Choice, where $\kappa$ is strongly inaccessible in $M^+$: $M^+\models \forall\, \vec{x}\, \exists y\,\varphi(\vec{x}, y)\to \forall x \exists\,\vec{z}(z_0=x\wedge\forall i<\kappa\,\varphi(\vec{z}\restriction i, z_i))$ where $\vec{x}, \vec{z}$ are $\kappa$-length sequences.  So assume that $M^+\models \forall\, \vec{x}\, \exists y\,\varphi(\vec{x}, y)$. From what we have show above, we know that $\vec{x}$ is an ordinal length sequence of elements in $\mathcal{M}$ and also $y$ is an element of $\mathcal{M}$ (as these can be classes we will write them with upper case letters in $\mathcal{M}$). Let $\psi$ be the second-order formula associated to $\varphi$, i.e. $\psi$ is the formula that says exactly the same as $\varphi$ only that its variables can be classes. Then by DC$_{\infty}$ we have that $\forall \vec{X}\exists Y \psi(\vec{X},Y)\to\forall X\exists\vec{Z}\,(Z_0=X\wedge\forall i\in ORD\,\psi(\vec{Z}\restriction i, Z_{i}))$
 where $\vec{X}, \vec{Z}$ are sequences of classes with ordinal length and $Z\restriction i$ is the sequence of the previously ``chosen'' $Z_j$, $j<i$. As before all the classes mentioned here are elements of $M^+$ where $\vec{Z}$ is a $\kappa$-length sequence and so we have proven the $\kappa$-Dependent Choice.

For 2.: Again we only have to proof the case of DC$_{\infty}$ and this is an direct analog to the proof of the Comprehension Axiom in the proof of Proposition \ref{conversemain}.

\end{proof}

\begin{lem}\label{distributive}
Let $M^+$ be a model of SetMK$^{**}$ with largest cardinal $\kappa$ and $P$ be an $M^+$-definable  class forcing notion. Then if $P$ is $\leq\kappa$-closed it is $\leq\kappa$-distributive.
\end{lem}

\begin{proof}
Let $p\in P$ and $\langle D_i\,\vert\, i<\beta\rangle$ is an $M^+$ definable sequence of dense classes, $\beta\leq\kappa$, and we want to show that there is a $q\leq p$ meeting each $D_i$ ($q$ meets $D_i$ if $q\leq q_i\in D_i$ for some $q_i$). As we have shown that $P$ is $\leq\kappa$-closed we want to construct a descending sequence $p_0\geq p_1\geq\ldots\geq p_i\geq\ldots$ $(i<\beta)$ with $p_i\in D_i$ for all $i<\beta$. Here we need the SetMK$^{**}$ version of the Dependent Choice Axiom we added to MK$^*$: Recall that $\kappa$-Dependent Choice says that $\forall \vec{x}\,\exists y \varphi(\vec{x},y)\to\forall x\exists\vec{z}\,(z_0=x\wedge\forall i<\kappa\,\varphi(\vec{z}\restriction i, z_{i}))$ where $\vec{x}$ is a $<\kappa$-length sequence of sets, $\vec{z}$ is a $\kappa$-length sequences of sets and $z\restriction i$ is the sequence of the previously ``chosen'' $z_j$, $j<i$. If we take $\varphi(\vec{x},y)$ to mean that ``$\vec{x}$ is a descending sequence of conditions, $x_i\in D_i$ for $i<\,\text{length}\,\vec{x}$, $y$ is a lower bound for $\vec{x}$ and $y\in D_{length_{\vec{x}}}$'' then we know that we can find a descending sequence $p_0\geq p_1\geq\ldots\geq p_i\geq\ldots$ $(i<\beta)$ with $p_i\in D_i$ for all $i<\beta$ such that there is an $q\in P$ with $q\leq p$ and $q\leq p_i$ for all $i<\beta$ and so $q$ meets all $D_i$. 
\end{proof}

\begin{theo}\label{prepforcing}
Let $M^+$ be a model of SetMK$^{**}$ with largest cardinal $\kappa$ and let $\kappa^*$ denote the height of M$^+$. Then there is an $M^+$-definable forcing $P$ such that the Definability Lemma holds and $P$ is  pretame, which adds a class predicate $A\subseteq\kappa^*$ such that $M^+=L_{\kappa^*}[A]$ and $(M^+, A)\models SetMK^{**}$ relativized to $A$.
\end{theo}
\begin{proof}
Let $P=\{\, p:\beta\to 2\,\vert\,\beta<\kappa^*,\, p\in M^+\}$ and let $G$ be $P$-generic over $M^+$. Let $\bigcup G=g:\kappa^*\to 2$ and $A=\{\gamma<\kappa^*\,\vert\, g(\gamma)=1\}$. Note that $G$ is an amenable predicate, i.e. $G\cap a$ belongs to $M^+$ for every $a\in M^+$ and $P$ is $\leq\kappa$-closed, as for every $\lambda\leq\kappa$ and every descending sequence $p_0\geq p_1\geq\ldots\geq p_i\geq\ldots$ $(i<\lambda)$ there is $q=\bigcup_{i<\lambda} p_i\in P$ such that $\forall i<\lambda\,\,q\leq p_i$.

To show that the forcing relation is definable in the ground model, we will concentrate on the atomic cases ``$p\Vdash\sigma\in\tau$'' and ``$p\Vdash\sigma=\tau$''. Then the other cases follow by induction. For $p\Vdash\sigma\in\tau$ first consider the case where the length of $p$ is larger then the ranks of $\sigma$ and $\tau$ (i.e. there is an $\gamma$ such that $rank\,\sigma, rank\,\tau<\gamma$ and $Dom(p)>\gamma$). Then the question if $\sigma^G\in \tau^G$ is already decided by $p$, meaning that $\sigma^G\in \tau^G$ exactly when $\sigma^p\in \tau^p$ with $\tau^p=\{\pi^p\,\vert\,\langle \pi, q\rangle\in\tau, p\leq q\}$ as $p$ ``has no holes'' and therefore a condition that extends $p$ will never change the decisions made below the length of $p$. This now defines the forcing relation because $P$ doesn't add any new sets and therefore $\sigma^p$ and $\tau^p$ are already elements of the ground model. If $p$ is not large enough to decide if $\sigma^G$ is an element of $\tau^G$, then we have to check that every $q$ that extends $p$ decides that this is the case so we get the definition ``$p\Vdash\sigma\in\tau\leftrightarrow\forall q\leq p\,(\,\vert q\vert\,>rank\,\sigma, rank\,\tau\to \sigma^q\in\tau^q)$''. The definitions for the ``$=$'' case can be given the same way and so the forcing is definable. The Truth Lemma then follows from Definability by the usual arguments.

Next we want to show that $P$ is pretame: As $P$ is $\leq\kappa$-closed, we know by Lemma \ref{distributive} that $P$ is $\leq\kappa$-distributive. Then $P$ is also pretame for sequences of dense classes of length $\leq\kappa$ and therefore $P$ is pretame. 

We have shown that $P$ doesn't add any new sets to the extension but a subclass $A\subset\kappa^*$. So the forcing just reorganizes $M^+$ and adds $A$ as a predicate. Then every set of ordinals from $M^+$ is copied into an interval of the generic and so every set of ordinals and therefore also every set is coded by $A$. Also as $A$ adds no new sets it holds that $L_{\kappa^*}[A]\subseteq M^+$. It follows that $M^+[G]=L_{\kappa^*}[A]$ and therefore already $M^+=L_{\kappa^*}[A]$. 

It remains to show that  $(M^+, A)\models (SetMK^{**})^A$, i.e. $SetMK^{**}$ holds for formulas which can mention $A$ as a predicate. As $P$ preserves the strongly inaccessibility of $\kappa$ it follows by Proposition \ref{forcingsetmk} that $M^+[G]\models SetMK^{*}$ and that means that $(M^+, A)\models SetMK^{*}$. But as the Comprehension and Bounding can mention the generic this implies that $(M^+, A)\models (SetMK^{*})^A$. For the DC$_{\kappa}$ note that by adding $A$ we now have a global well-order of the extension. That means that if we have a $<-\kappa$ sequence $\vec{x}$ in $M^+[G]$ such that $\forall \vec{x}\,\exists y\,\varphi(\vec{x}, y)$ and we want to find a $\kappa$-length sequence $\vec{z}$ such that $\forall x\exists\vec{z}\,(z_0=x\wedge\forall i<\kappa\,\varphi(\vec{z}\restriction i, z_{i}))$ we can just take $z_i$ to be least so that $\varphi(\vec{z}\restriction i, z_i)$ for each $i$.
\end{proof}
 
 As our ultimate goal is to go back to an MK$^{**}$ model, we want to show that the predicate $A$ can be coded into a subset of $\kappa$:
\begin{theo} \label{Jensencoding}
Let $(M^+, A)$ be a model of SetMK$^{**}$ relativized to a predicate $A$, with largest cardinal $\kappa$ and let $\kappa^*$ denote the height of $(M^+, A)$, where $A$ is the generic predicate added by the forcing $P$ in Theorem \ref{prepforcing} and $M^+=L_{\kappa^*}[A]$. Then we can force that there is a $X\subseteq \kappa$ such that $L_{\kappa^*}[A]\subseteq L_{\kappa^*}[X]$, SetMK$^{**}$ is preserved and $\kappa$ remains strongly inaccessible.
\end{theo}
\begin{proof}
To get $A$ definable in $M^+[X]$, for some $X\subseteq\kappa$, we want to use an almost disjoint forcing which codes the predicate $A$ into such an $X$. The forcing will be along the following lines: we will need to define a family $S$ of almost disjoint sets (i.e. for $x, y\subseteq \kappa$, $x$ and $y$ are almost disjoint if $x\cap y$ is bounded in $\kappa$) $A_{\beta}$ which we will use to code the predicate $A\subseteq\kappa^*$ into an $X$. We will define $A_{\beta}$ to be the least subset of $\kappa$ (i.e. least in the canonical well-order of $L_{\kappa^*}[A\cap \beta]$) in $L_{\kappa^*}[A\cap \beta]$ which is distinct from the $A_{\bar{\beta}}$ for $\bar{\beta}<\beta$. The idea is that we can decode $A$ in $L_{\kappa^*}[X]$ if we know the $A_{\beta}$'s. But as $A$ is a proper class we don't know that we can always find such distinct $A_{\beta}$'s. So we will have to assume that the cardinality of $\beta$ is at most $\kappa$ not only in $L_{\kappa^*}[A]$ but also in $L_{\kappa^*}[A\cap\beta]$ because now to find an $A_{\beta}$ distinct from each $A_{\bar{\beta}}$, $\bar{\beta}<\beta$, we can list these $A_{\bar{\beta}}$'s as $\langle A_i\,\vert\, i<\kappa\rangle$ and obtain $A_{\beta}$ by diagonalization. To fulfill that assumption however we have to ``reshape'' $A$ into a predicate $A'$ that has the property that if $\beta<\kappa^*$ then the cardinality of $\beta$ is $\leq\kappa$ in $L_{\kappa^*}[A'\cap \beta]$. Then we can code $A$ as the even part of $A'$ to get $(M^+, A')\models (SetMK^{**})^{A'}$ and finally code $A'$ by a subset of $\kappa$.

So the proof consists of two steps: First we have to show that we can reshape $A$ and then we have to force with an almost disjoint forcing to show that the reshaped predicate $A'$ can be coded into a subset of $\kappa$, preserving SetMK$^{**}$ in each step.

Step 1: We add a reshaped predicate $A'$ over $(L_{\kappa^*}[A], A)$ by the following forcing:
\begin{equation*}
P=\{p:\beta\to 2\,\vert\, \kappa\leq\beta<\kappa^*, \forall\,\gamma\leq\beta\, (L_{\kappa^*}[A\cap\gamma, p\restriction\gamma]\vDash\text \,\vert\,\gamma\,\vert \leq\kappa)\}
\end{equation*}
The main obstacle is to show that $P$ is definably-distributive, i.e. we have to show that for a $p\in P$ and $(M^+, A)$-definable sequences of dense classes of set-length $\langle D_i\,\vert\, i<\alpha\rangle$ for all $\alpha\leq\kappa$, there is a $q\leq p$ meeting each $D_i$ with $q\in P$.
\begin{claim} \label{distributiveP}
$P$ is definably-distributive. 
\end{claim}
\begin{proof}
Note that it suffices to show definable-distributivity for $\kappa$; so we consider an $(M^+, A)$-definable sequence of dense classes $\langle D_i\,\vert\, i<\kappa\rangle$. We want to define a descending sequence of conditions $p\geq p_0\geq p_1\geq\ldots$ where $p_i\geq q$, $q\in P$ and $p_{i+1}\in D_i$ for each $i <\kappa$. To show that the $p_i$ are indeed conditions we have to show that $L_{\kappa^*}[A\cap\gamma, p_i\restriction\gamma]\models\,\vert\gamma\vert\leq\kappa$ for every $\gamma\leq\vert p_i\vert$. In the following we will use the fact that a condition is always extendible to any length $<\kappa^*$:
$\forall p\,\forall\beta<\kappa^*\,\exists q\leq p, \vert q\vert\geq\beta, q\in P$. This holds because there is an $x\subseteq\kappa$ such that $\beta$ is coded by $x$ and $p*x\in P$ and has length $\vert p\vert +\kappa$. If this is still below $\beta$ we can lengthen $p$ further by a sequence of $0$'s: $q=p*x*\vec{0}$. This will again be an element of $P$ as we know from the information in the code $x$ of $\beta$ that the ordinals will collapse.

First, we assume that the sequence of dense classes is $\Sigma_1$-definable, i.e. $\{(q, i)\,\vert\, q\in D_i\}$ is $\Sigma_1$-definable with parameter.  

As we have seen that every condition is extendible, we can extend $p$ to catch a parameter $x\in L_{\vert p\vert}[A]$ such that the sequence of the $D_i$ is $\Sigma_1$-definable with parameter $x$. Let $p_0$ be this extension of $p$. Then, as we have Global Choice, we can consider the $<_{(M^+, A)}$-least pair $(q_0, w_0)$ such that $q_0\leq p_0$ and $w_0$ witnesses ``$q_0\in D_0$''. 
Then we choose $p_1$ such that $p_1$ is a condition which extends $q_0$ such that $w_0\in L_{\vert p_1\vert}[A\cap\vert p_1\vert]$. Now we define $p_2$ in the same way: Choose $(q_1, w_1)$ such that $q_1\leq p_1$ and $w_1$ witnesses ``$q_1\in D_1$''. Then let $p_2\leq q_1$ such that $w_1\in L_{\vert p_2\vert}[A\cap\vert p_2\vert]$. Define the rest of the successor cases $(p_{n+1}, w_{n+1})$ similarly.

For the first of the limit cases, let $p_{\omega}=\bigcup_{n<\omega} p_n$ and we claim that $p_{\omega}\in P$. So we have to show that $\forall \gamma\leq\vert p_{\omega}\vert$, $\gamma$ collapses to $\kappa$ using only $A\cap\gamma$ and $p_{\omega}\restriction\gamma$. We know that if $\gamma<\vert p_{\omega}\vert$ then $\gamma<\vert p_n\vert$ for some $n$. So we only have to consider the case where $\gamma=\vert p_{\omega}\vert$. It follows from the construction of the $p_n$'s that the sequence $\langle p_n\,\vert\, n<\omega\rangle$ is definable over $L_{\vert p_{\omega}\vert}[A\cap\vert p_{\omega}\vert, p_{\omega}]$ and is a cofinal sequence in $p_{\omega}$, i.e. it converges to $p_{\omega}$. Then also the sequence of the lengths of the $p_n$'s, $\langle\,\vert p_n\vert \,\vert\, n<\omega\rangle$ is definable over $L_{\vert p_{\omega}\vert} [A\cap\vert p_{\omega}\vert, p_{\omega}]$ and converges to $\vert p_{\omega}\vert$. As we know that $\vert p_n\vert$ collapses to $\kappa$ for every $n<\omega$, we know that in $L_{\vert p_{\omega}\vert}[A\cap\vert p_{\omega}\vert, p_{\omega}]$ $\vert p_{\omega}\vert$ definably collapses to $\kappa$. So $L_{\vert p_{\omega}\vert +1}[A\cap\vert p_{\omega}\vert, p_{\omega}]\models \vert p_{\omega}\vert$ is collapsed to $\kappa$. The other limit cases can be handled in the same way.

Now we go to the $\Sigma_2$-definable case. Note that we cannot simply copy the construction of the $p_n$-sequence because the witness $q_{n+1}$ we need for the definition of the next $p_{n+1}$ will now be a solution to a $\Pi_1$-statement and will therefore not be absolute in the other models. But we know that for $V=L_{\kappa^*}[A]$ it holds that $\forall\alpha<\kappa^*$ $\exists\beta\leq\kappa^*$, $\alpha<\beta$ such that $L_{\beta}[A]$ is $\Sigma_n$-elementary in $L_{\kappa^*}[A]$. This holds because for a pair $\alpha, n$ we can take the $\Sigma_n$-Skolem Hull $N$ of $\alpha$ in $L_{\kappa^*}[A]$. Then in $M$ we have a solution for every $\Sigma_n$-property with parameters $<\alpha$, $M$ is transitive and bounded by Class-Bounding. Then there is a $\beta\leq\kappa^*$ such that $M$ is equal to $L_{\beta}[A]$.

So we can always find models that are $\Sigma_1$-elementary submodels of $(M^+, A)$ in which we can carry out the definition of the sequence of conditions: As before we choose for every $n<\omega$ a pair $(q_{n}, w_{n})$ such that $q_{n}\leq p_{n}$ such that $w_{n}$ witnesses ``$q_{n}\in D_{n}$'' and then let $p_{n+1}\leq q_{n}$ such that $w_{n}\in L_{\vert p_{n+1}\vert}[A\cap\vert p_{n+1}\vert, p_{n+1}]$ and $L_{\vert p_{n+1}\vert}[A\cap\vert p_{n+1}\vert, p_{n+1}]$ is an $\Sigma_1$-elementary submodel of $L_{\kappa^*}[A]$. This also holds in the limit case by using the same construction we did for the $\Sigma_1$ case where again the model $L_{\vert p_{\omega}\vert+1}[A\cap\,\vert p_{\omega}\vert, p_{\omega}]$ is an $\Sigma_1$-elementary submodel of $L_{\kappa^*}[A]$. The same can be done for all the $\Sigma_m$-definable cases.
\end{proof}

Now that we know that $P$ is $\leq\kappa$-distributive, we know that $P$ is $\leq\kappa$-pretame and therefore $(M^+, A, A')\models (SetMK^{**})^{A,A'}$ (similar to proof of Theorem \ref{prepforcing} by using Proposition \ref{forcingsetmk} and the fact that there is a global well-order of the extension). Then we can code $A$ to be the even part of $A'$ and we get a model $(M^+, A')\models (SetMK^{**})^{A'}$. It remains to show that $A'$ can be coded into a subset of $\kappa$.

Step 2: Code $A'$ into $X\subseteq \kappa$. As we know that $A'$ is reshaped we can define a collection of sets $\mathcal{S}=\langle A_\beta\,\vert\,\beta<\kappa^*\rangle$ in the following way: let $A_{\beta}$ be the least $B\subseteq\kappa$ in $L_{\kappa^*}[A'\cap\beta]$ such that $B\notin\{A_{\bar{\beta}}\,\vert\,\bar{\beta}<\beta\}$. $\mathcal{S}$ can be turned into a collection $\mathcal{S}'=\langle A'_\beta\,\vert\,\beta<\kappa^*\rangle$ of almost disjoint sets $A'_{\beta}$ by mapping every set to the set of codes of its proper initial segments: $B\subseteq\kappa$ is mapped to $B'=\{\text{Code}\,(B\cap\alpha)\,\vert\,\alpha<\kappa\}\subseteq\kappa$. Then for two distinct subsets $B$ and $C$ of $\kappa$, $\vert\,B'\cap C'\,\vert\,<\kappa$ and therefore they are almost disjoint. We want to show that we can code $A'$ by a subset $X$ of $\kappa$ by showing that $X\cap A'_{\beta}$ is bounded if and only if $\beta\in A'$. This can be done by a forcing $Q$ with the conditions $(g,S)$ where $S\subseteq A'$, $\vert S\vert<\kappa$ and $g$ is an element of $^{<\kappa}2$. Extension is defined by: $(g, S)\geq (h,T)$ iff $h$ extends $g$, $S\subseteq T$ and if $\beta\in S$ and $h(\gamma)=1$ for a $\gamma\in A'_{\beta}$ then $g(\gamma)=1$. Note that two conditions with the same first component $\langle g, S\rangle$ and $\langle g, T\rangle$ are compatible because we can always find a common extension $\langle g, S\cup T\rangle$. Thus a function which maps every element of a definable antichain into its first component is injective (as otherwise the conditions would be compatible). So we have injectively mapped a definable class to a set as there are only $\kappa$ many first components. By Bounding such a function exists as a set and so $Q$ is set-c.c., i.e. every definable antichain is only set-sized. Then $Q$ is pretame, as every definable dense class can be seen as an antichain. 
Now let $G$ be a $Q$-generic, $G_0=\bigcup\{ g\,\vert\, (g,S)\in G\}$ and $X=\{\gamma\,\vert\, G_0(\gamma)=1\}$. we argue that we can find the almost disjoint sets in $L_{\kappa^*}[X]$ because $A'$ is reshaped and therefore it holds for any $\beta$ that $\vert\beta\vert\,\leq\kappa$ in $L_{\kappa^*}[A'\cap\beta]$. So after $X$ has decoded $A'\cap\beta$ it can find $A'_{\beta}$ and then continue the decoding in the following way: $\beta\in A'$ if there is an $(g, S)\in G$ with $\beta\in S$ and by the definition of extension if $G_0(\gamma)=1$ for a $\gamma\in A'_{\beta}$ then $g(\gamma)=1$. So $X\cap A'_{\beta}=\{\gamma\,\vert\, g(\gamma)=1\}\cap A'_{\beta}$ and that is bounded and therefore we have a code of $A'$ by $X$ via
\begin{equation*}
X\cap A'_{\beta}\text{ is bounded if and only if }\beta\in A'.
\end{equation*}

As this forcing is $\kappa$-closed (i.e. closed for $<\kappa$ sequences), $\kappa$ stays regular and therefore strongly inaccessible and by Proposition \ref{forcingsetmk} SetMK$^{*}$ is preserved and by Proposition \ref{forcingsetmk} SetMK$^*$ is preserved.
\end{proof}

We have seen how definable hyperclass-forcing can be carried out over a model $\mathcal{M}$ of MK$^{**}$: First we go to the related SetMK$^{**}$ model $M^+$ (Theorem \ref{maintheorem}). Then in order to be able to force over this model, we change $M^+$ to a model $L_{\kappa^*}[A]$ for a generic predicate $A$ (Theorem \ref{prepforcing}). Finally we showed how to code $A$ into a subset $X\subseteq\kappa$ to avoid having an undefinable predicate once we go back to the extension of the original MK$^{**}$ model (Theorem \ref{Jensencoding}). At this point we can force with any desirable pretame definable class-forcing over $L_{\kappa^*}[X]$, go back to MK$^{**}$ and get the desired definable hyperclass-forcing over MK$^{**}$. 

So we have given a template which allows us to do definable hyperclass-forcing over MK$^{**}$. In the following we will show how to use this template to produce minimal $\beta$-models of MK$^{**}$.

\section{Minimal $\beta$-Models of MK$^{**}$}

As an application of definable hyperclass forcing we will show that every $\beta$-model of MK$^{**}$ can be extended to a minimal $\beta$-model of MK$^{**}$ via the use of SetMK$^{**}$ models. Here a minimal model $M(S)$ of SetMK$^{**}$ is the least transitive model of SetMK$^{**}$ containing a real $S$ and equivalently a minimal $\beta$-model $\mathcal{M}(S)$ of MK$^{**}$ is the least $\beta$-model of MK$^{**}$ containing a real $S$.\footnote{We can see here that it is vital to restrict ourselves to $\beta$-models in order to talk about minimal models of MK by comparing this to the situation in ZFC: There it also only makes sense to talk about minimal models containing a real for well-founded models (and not for ill-founded models). So by making the transformation from MK to SetMK we have to restrict ourselves to $\beta$-models.} For that we will use and modify the template developed in the last section: We start with an arbitrary $\beta$-model $\mathcal{M}=(M, \mathcal{C})$ of MK$^{**}$ and from that we get the corresponding model $M^+$ of SetMK$^{**}$ (by Theorem \ref{maintheorem}) with $M=V_{\kappa}^{M^+}$ and $\mathcal{C}=P(M)\cap M^+$ where $\kappa$ is strongly inaccessible in $M^+$. Let $\kappa^*$ denote the height of $M^+$ and apply Theorem \ref{prepforcing} to arrive at $M^+=L_{\kappa^*}[A]$ where $A\subseteq \kappa^*$  and $(M^+, A)$ satisfies SetMK$^(**)$ relative to $A$. We now show that we can extend $M^+$ to a minimal model of SetMK$^{**}$ and then go back to an MK$^{**}$ model, which will be a minimal $\beta$-model of MK$^{**}$.

\begin{theo}\label{minimalmodels}
Every $\beta$-model of MK$^{**}$ can be extended to a minimal $\beta$-model of MK$^{**}$ with the same ordinals.
\end{theo}

\begin{proof}
First we use the template described above to arrive at the model $L_{\kappa^*}[A]$ and then we will code the predicate $A$ into a subset of $\kappa$ by using Theorem \ref{Jensencoding} with a small modification in the ``reshaping'' forcing. Instead of forcing that each $\gamma<\kappa^*$ collapses in $L_{\kappa^*}[A\cap\gamma, p\restriction\gamma]$, we will force it to already collapse instantly in the next level, i.e. in $L_{\gamma+1}[A\cap\gamma, p\restriction\gamma]$. So the forcing will be:
\begin{equation*}
P=\{p:\beta\to 2\,\vert\, \kappa\leq\beta<\kappa^*, \forall\,\gamma\leq\beta\, (L_{\gamma+1}[A\cap\gamma, p\restriction\gamma]\vDash\text \,\vert\,\gamma\,\vert \leq\kappa)\}
\end{equation*}
The proof that $P$ is definably-distributive then works in exactly the same way. As in Theorem \ref{Jensencoding} we can code $A$ to be the even part of the predicate $A'$ added by the reshaping forcing which in turn can be coded into an $X\subseteq\kappa$ by an almost disjoint forcing. This gives us that there are no SetMK$^{**}$ models containing $X$ of height between $\kappa$ and $\kappa^*$: In the reshaping forcing we destroyed the Replacement axiom level by level relative to $A$ and in the almost disjoint coding we can now choose the codes instantly level-by-level (i.e. every code for $\gamma$ appears in $L_{\gamma+1}[X]$). So $A'$ can be recovered level-by-level from $X$ and therefore Replacement is also destroyed level-by-level relative to $X$. We arrive at a SetMK$^{**}$ model $L_{\kappa^*}[X]$, with $X\subseteq\kappa$, which is the least transitive $ZFC^-$ model containing $X$ (again $\kappa$ remains regular and indeed strongly inaccessible, because the almost disjoint coding is $\kappa$-closed).

We will extend this to a minimal model of SetMK$^{**}$ in two steps: First we extend $L_{\kappa^*}[X]$ to a model $L_{\kappa^*}[Y]$ such that no cardinal $\bar{\kappa}<\kappa^*$ can serve as a ``source'' for a SetMK$^{**}$ model (i.e. is the largest cardinal of a Set$MK^{**}$ model containing $Y\cap\bar{\kappa}$) and second we show that we can add a real $S$ such that in $L_{\kappa^*}[S]$ there are no SetMK$^{**}$ models containing $S$ below $\kappa^*$. Then it only remains to show that from $L_{\kappa^*}[S]$ we can go back to a minimal $\beta$-model of MK$^{**}$.\\

Step 1: With the modification of Theorem \ref{Jensencoding}, we have shown that there are no SetMK$^{**}$ models containing $X$ between $\kappa$ and $\kappa^*$. But it could still be that there exist  cardinals below $\kappa$ which are sources for SetMK$^{**}$ models. We will destroy these cardinals by shooting a club through a ``fat-stationary'' set which has no such cardinals and then force all limit cardinals to belong to this club. 

So let $S=\{\bar{\kappa}<\kappa\,\vert\,\bar{\kappa}$ is a limit cardinal and for all $\bar{\beta}>\bar{\kappa}$, if  $L_{\bar{\beta}}[X\cap\bar{\kappa}]\vDash ZFC^-$ then $L_{\bar{\beta}}[X\cap\bar{\kappa}]\nvDash\bar{\kappa}$ is strongly inaccessible$\}$. 
\begin{defi}
$S$ is \emph{fat-stationary} if for every club $C$ in $L_{\kappa^*}[X]$, $S\cap C$ contains closed subsets of any order type less than $\kappa$. 
\end{defi}
We prove the following:
\begin{lem}
$S$ is fat-stationary and there is a $\kappa$-distributive (i.e. $<\kappa$ distributive) forcing of size $\kappa$ that adds a club $C\subseteq S$.
\end{lem}
\begin{proof}
First we will show that $S$ is stationary with respect to clubs in $L_{\kappa^*}[X]$. So suppose $C$ is a club in $L_{\alpha}[X]$ for an $\alpha<\kappa^*$. We build an increasing sequence $\langle M_n\,\vert\, n<\omega\rangle$ of sufficiently elementary submodels of $L_{\alpha}[X]$ in the following way: Let $M_0$ be the $\Sigma_1$-Skolem Hull of $\omega\cup\{X, C\}$ in $L_{\alpha}[X]$. Then $C\in M_0$ and $\kappa_0= sup(M_0\cap\kappa)$ is a cardinal. Next, let $M_1$ be the $\Sigma_1$-Skolem Hull of $\kappa_{0} +1\cup\{X, C\}$ in $L_{\alpha}[X]$ and $\kappa_1= sup(M_1\cap\kappa)$. Repeat this construction for all $n<\omega$. Then this sequence of elementary submodels is definable over $M_{\omega}=\bigcup_{n<\omega} M_n$ and $\kappa_{\omega}=sup_{n<\omega}\kappa_n <\kappa$ is a cardinal in $C$ as $C$ is closed, unbounded in $\kappa$. Also $\kappa_{\omega}$ is an element of $S$ because if $L_{\bar{\alpha}}[X\cap\kappa_{\omega}]$ is the transitive collapse of $M_{\omega}$ then there are no ZFC$^-$ models containing $X\cap\kappa_{\omega}$ of height $<\bar{\alpha}$ (by elementarity), of height $=\bar{\alpha}$ because $\langle \kappa_n\,\vert\, n<\omega\rangle$ is definable over it (and so $\kappa_{\omega}$ becomes definably singular) and any ZFC$^-$ model containing $X\cap\kappa_{\omega}$ of height $>\bar{\alpha}$ sees that $\kappa_{\omega}$ has cofinality $\omega$ (as the $\kappa_n$-sequence is an element of it).

To show that $S$ is fat-stationary we can use the same proof as for stationarity except one uses a longer $\delta$-sequence of elementary submodels, for $\delta$ a limit cardinal less than $\kappa$. 

Now for the second part of the Lemma we can force with a set-forcing to add a club. Here we will closely follow the proof of the ZFC version of this claim, as proven in \cite{Abraham1983-ABRFCU} (see there for more details). Let $Q=\{ p\,\vert\, p\text{ is a closed, bounded subset of } S\}$ be a forcing notion ordered by end-extensions: $q\leq p$ iff $p=q\cap(\text{sup}(p)+1)$. For $G$ $Q$-generic over $L_{\kappa^*}[X]$ let $C=\bigcup G$. Then $C$ is closed and unbounded and a subset of $S$. To show  that $Q$ is $\kappa$-distributive we have to show that for every $\tau<\kappa$ and sequence $\mathcal{D}=\langle D_i \,\vert\, i\in\tau\rangle$ of open, dense subsets of $Q$, $\bigcap_{i<\tau} D_i$ is dense in $Q$. Now we can define a sequence of elementary substructures $\langle M_{\alpha}\,\vert\,\alpha<\kappa\rangle$ of $L_{\kappa^*}[X]$ such that $c_{\alpha}=M_{\alpha}\cap\kappa$ is an ordinal and $\langle c_{\alpha}\,\vert\,\alpha<\kappa\rangle$ is an increasing and continuous sequence cofinal in $\kappa$. Let $E$ be the collection of the $c_{\alpha}$, $\alpha<\kappa$. Because $S$ is fat-stationary, $S\cap E$ contains a closed subset $A$ of order-type $\tau +1$. Then in the model $M_{\alpha}$, with $\alpha=\text{sup}(A)$, we can define an increasing sequence $\langle p_i\,\vert\, i<\tau\rangle$, such that $p_i\in Q$ and $p_{i+1}\in D_i\cap M_{\alpha}$. We can define $p_{\tau}=\bigcup_{i<\tau} p_i\cup\{\alpha\}$ and this will be in $\bigcap_{i<\tau} D_i$. Note that this (set-) forcing is an element of $L_{\kappa^*}[X]$ and therefore preserves ZFC$^-$. Furthermore, as this forcing doesn't add sets of size $<\kappa$, $\kappa$ stays strongly inaccessible and SetMK$^*$ is preserved because of Proposition \ref{forcingsetmk}.
\end{proof}

Let $X'$ be the join of $X$ with the club we added. Then $X'\subseteq\kappa$ and the resulting model is $L_{\kappa^*}[X']$.

\begin{lem}
We can force all limit cardinals to belong to $C$ with a forcing of size $\kappa$ such that $\kappa$ remains strongly inaccessible.
\end{lem}
\begin{proof}
Enumerate $C$ as follows: $C=\langle \bar{\kappa}_i\,\vert\,i<\kappa\rangle$. We may assume that each $\bar{\kappa}_i$ is a strong limit cardinal (as $\kappa$ is strongly inaccessible we can thin out $C$). Then we can build an Easton product of collapses, where we collapse every $\bar{\kappa}_{i+1}$ to the successor of $\bar{\kappa}_i$ and therefore ensure that all limit cardinals below $\kappa$ are limits of cardinals in $C$ and therefore are themselves in $C$.

So for $i<\kappa$ consider $Col_i(\bar{\kappa}_i^+, \bar{\kappa}_{i+1})$, where the conditions are functions $p$ with $dom(p)\subset\bar{\kappa}_i^+$, $\vert dom(p)\vert<\bar{\kappa}_i^+$ and $range(p)\subset\bar{\kappa}_{\i+1}$. Cardinals below $\bar{\kappa}_i^+$ and above $\bar{\kappa}_{i+1}^{\bar{\kappa}_i}$ are preserved (the size of the forcing is $\bar{\kappa}_{i+1}^{\bar{\kappa}_i}$) and in the extension we have a function which maps $\bar{\kappa}_i^+$ onto $\bar{\kappa}_{i+1}$.

Now we can build the Easton product (product with Easton support) of these collapses for every $i<\kappa$: A condition $p$ in this forcing is a function such that $p=\langle p_i\,\vert\, i<\kappa\rangle\in \Pi_{\i<\kappa}\text{Col}_i(\bar{\kappa}_i^+, \bar{\kappa}_{i+1})$ and the forcing is ordered by end-extension. $p$ has Easton support, i.e. for every inaccessible cardinal $\lambda$, $\vert\,\{\alpha<\lambda\,\vert\, p(\alpha)\neq\emptyset\}\,\vert<\lambda$. 
As usual with Easton Products the forcing notion $P$ can be split into two parts $P(\leq\lambda)=\Pi_{\i\leq\lambda}\text{Col}_i(\bar{\kappa}_i^+, \bar{\kappa}_{i+1})$ and $P(>\lambda)=\Pi_{\lambda<\i<\kappa}\text{Col}_i(\bar{\kappa}_i^+, \bar{\kappa}_{i+1})$ for every regular cardinal $\lambda$. For this reason and as each $\bar{\kappa}_i$ is a strong limit, each collapse from $\bar{\kappa}_{i+1}$ to $\bar{\kappa}_i^+$ will not be affected by the other collapses and $\kappa$ remains regular and strong limit. Furthermore, as this forcing is in $L_{\kappa^*}[X']$ (it is of size $\kappa$) it preserves SetMK$^{**}$.

Because of the unboundedness of $C$, every limit cardinal is also a limit of cardinals in $C$ and therefore, as $C$ is closed, it is an element of $C$.
\end{proof}
We conclude Step 1 by choosing $X''$ to be the join of $X'$ and the above Easton product. Then we arrive at a model $L_{\kappa^*}[X'']$ with $X''\subseteq\kappa$ such that for every cardinal $\bar{\kappa}<\kappa^*$ there is no model of ZFC$^-$ containing $X''\cap\bar{\kappa}$ in which $\bar{\kappa}$ is inaccessible and therefore $\bar{\kappa}$ is not a source for a SetMK$^{**}$ model.\\

Step 2: We want to extend the results from the last step to hold for all ordinals, i.e. for all ordinals $\alpha<\kappa^*$ there is no SetMK$^{**}$ model of height $<\kappa^*$ containing a real $S$ in which $\alpha$ is strongly inaccessible. This makes use of Jensen coding and a result about admissibility spectra which is connected to it. We will use these results as black boxes and will only state the main definitions and theorems here:
\begin{theo}[Jensen Coding]
Suppose that $\langle M, A\rangle$ is a transitive model of ZFC, i.e. $M$ is a transitive model of ZFC, $A\subseteq M$ and Replacement holds in $M$ for formulas mentioning $A$ as a unary predicate. Then there is an $\langle M, A\rangle$-definable class forcing $P$ such that if $G\subseteq P$ is $P$-generic over $\langle M, A\rangle$, then:
\begin{itemize}
\item[a)] $\langle M[G], A, G\rangle\models ZFC$.
\item[b)] For some $R\subseteq \omega$, $M[G]\models V=L[R]$ and $\langle M[G], A, G\rangle\models A, G$ are definable from the parameter $R$.
\end{itemize}
\end{theo}
The very elaborate proof of this result uses Jensen's fine structure theory and, very roughly, the forcing involved consists of three components: an almost disjoint coding at successor cardinals, a variation thereof at limit cardinals and a reshaping forcing.\footnote{An detailed account of this can be found in \cite{beller1982coding}, a simplified version of the proof can be found in \cite{Friedman:2000}.}
\begin{defi}
Let $T$ be the theory of ZF without Power Set and with Replacement restricted to $\Sigma_1$ formulas. Then $\Lambda(R)$ for a real $R$ denotes the admissibility spectrum of $R$ and is defined as the class of all ordinals $\alpha$ such that $L_{\alpha}[R]\models T$, i.e. the class of all $R$-admissible ordinals.
\end{defi}
\begin{theo}[S.-D. Friedman]\label{spectrum}\footnote{See \cite {Friedman:2000}, Theorem 7.5, p. 142.}
Suppose $\varphi$ is $\Sigma_1$ and $L\models\varphi(\kappa)$ whenever $\kappa$ is an $L$-cardinal. Then there exists a real $R<_L 0^{\sharp}$ such that $\Lambda(R)\subseteq\{\alpha\,\vert\, L\models\varphi(\alpha)\}$ and $R$ is cardinal preserving over $L$.
\end{theo}
We will use these theorems to prove the following lemma:
\begin{lem}
We can extend the model $L_{\kappa^*}[X'']$ to be of the form $L_{\kappa^*}[S]$ for a real $S$ such that $L_{\kappa^*}[S]\models SetMK^{**}$ and whenever $\bar{\alpha}<\kappa^*$ is an ordinal  there is no model of SetMK$^{**}$ of height $<\kappa^*$ containing $S$ in which $\bar{\alpha}$ is strongly inaccessible.
\end{lem}
\begin{proof}
First we add a real $R$ to the resulting model of Step 1 and get a model $L_{\kappa^*}[R]\models SetMK^{**}$. 
This can be done by using Jensen coding over the model $L_{\kappa^*}[X'']$. Although we start from a model of ZFC$^-$ rather than ZFC our model is of the form $L_{\kappa^*}[X'']$ and therefore we can use the standard pretameness argument for Jensen coding to show that ZFC$^-$ is preserved\footnote{See \cite{Friedman:2000}, Chapter 4.}. Also, $\kappa$ will still be inaccessible in the extension because Jensen coding preserves inaccessibles.\footnote{This follows from an property called diagonal distributivity (see \cite{Friedman:2000}, p. 37).}  
 Note that the result from Step 1 still holds: In $L_{\kappa^*}[R]$ we have that if $\bar{\kappa}<\kappa^*$ is a cardinal then there is no transitive model of SetMK$^{**}$ containing $R$ in which $\bar{\kappa}$ is inaccessible as otherwise there would have been such a model containing $X''\cap\bar{\kappa}$ as the latter is coded by $R$ in $L_{\bar{\kappa}}[R]$.

Now we use Theorem \ref{spectrum} relativized to the real $R$ to produce a new real $S$ such that this holds for ordinals $\bar{\kappa}$. Theorem \ref{spectrum} works in the context of ZFC$^-$ for the same reasons as for Jensen coding. Note that $L_{\kappa^*}[R]\models\varphi(\bar{\kappa})$ for every $L_{\kappa^*}[R]$-cardinal $\bar{\kappa}$ where $\varphi(\alpha)$ is the following $\Sigma_1$ property with parameter $R$:  
``Either $L_{\alpha}[R]\models$ there is a largest cardinal or there is $\beta>\alpha$ such that $L_{\beta}[R]\models$ $\alpha$ is singular and for all $\gamma$ with $\alpha<\gamma<\beta$, $L_{\gamma}[R]\nvDash ZFC^-$''. This property says that either $\alpha$ is a successor or we can ``see'' the singularity of $\alpha$ before we see a ZFC$^-$ model for which it could be a source. Then by Theorem \ref{spectrum} there exists a real $S$ generic over $L_{\kappa^*}[R]$ such that $L_{\kappa^*}[S]\models SetMK^{**}$ and $\Lambda(S)\subseteq\{\alpha\,\vert\, L[R]\models\varphi(\alpha)\}$. As $\alpha$ which is inaccessible in a model of ZFC$^-$ containing $S$ is $S$-admissible, we get the desired property for all ordinals.
\end{proof}

We now have a minimal model $L_{\kappa^*}[S]$ of SetMK$^{**}$, i.e. the least transitive model of SetMK$^{**}$ containing $S$. It only remains to show that by going back to MK$^{**}$ we arrive at a minimal $\beta$-model of $MK^{**}$. To see that consider the model $(L_{\kappa}[S], \mathcal{C})$ where $\mathcal{C}$ consists of the subsets of $L_{\kappa}[S]$ in $L_{\kappa^*}[S]$. This is a $\beta$-model of MK$^{**}$ by Proposition \ref{conversemain} and it is the least such model containing $S$ because otherwise there exists a $\beta$-model $(N, \mathcal{C}')\subset (L_{\kappa}[S], \mathcal{C})$, $(N, \mathcal{C}')\models MK^{**}$ containing $S$ that would give rise to a model $N^+$ of SetMK$^{**}$. If we then go to the $L[S]$ of $N^+$ we arrive at a model $L_{\alpha}[S]$ for some $\alpha<\kappa^*$ which is a model of SetMK$^{**}$. This is a contradiction to the minimality of $L_{\kappa^*}[S]$.
\end{proof}

\section{Further Work and Open Questions}
These results opens up a wider area of further research and related open questions.

In the definition of definable hyperclass forcing we used the restriction to $\beta$-models of MK$^*$ to make the coding of a transitive SetMK$^*$ model work. It would be interesting to investigate what happens if we drop this restriction:\\

\emph{Question}: How can definable hyperclass forcing be defined for an arbitrary model of MK$^{**}$?\\

Dropping the $\beta$-model assumption for the coding would mean to work only internally in the MK$^{**}$ model and restricting ourselves to just coding pairs. We are confident that this can be done, but there are many details to be worked out.

In this paper we consider three variants of the axioms of Morse-Kelley; the standard form MK, the extension via Class-Bounding, here called MK$^*$ and the additional extension with Dependent Choice, called MK$^{**}$. The obvious question presents itself, which is how they are related:\\

\emph{Question}: Assuming just the consistency of MK, are there models of MK that don't satisfy MK$^*$ and models of MK$^*$ that don't satisfy MK$^{**}$?\\

Another fruitful topic is the analogy between Morse-Kelley and second-order arithmetic. \\

\emph{Question}: What results and questions can be transferred from the context of Morse-Kelley class theory to second-order arithmetic and vice versa?\\

As an example for this transfer, let us consider the question of minimal $\beta$-models of MK$^{**}$. It can be translated to minimal $\beta$-models of second-order arithmetic (plus Dependent Choice) in the following way:
\begin{theo}
Every $\beta$-model of second-order arithmetic with Dependent Choice can be extended to a minimal $\beta$-model of second-order arithmetic with Dependent Choice with the same ordinals.
\end{theo}
\begin{proof}[Proof outline]
Starting with a $\beta$-model of second-order arithmetic we can go to a related model of ZFC$^-$ where the inaccessible cardinal $\kappa$ is now simply $\aleph_0$. Then the question about models below the largest cardinal becomes trivial and we can concentrate on the case of eliminating models of ZFC$^-$ between $\aleph_0$ and the height of the model $\alpha^*$. First we change the ZFC$^-$ model to a model $L_{\alpha^*}[A]$ in a way analogous to Theorem \ref{prepforcing}. Then we can adapt the proof of Theorem \ref{Jensencoding} in a similar way as we did in the beginning of the proof of Theorem \ref{minimalmodels}: we code the predicate $A$ into $X\subseteq\aleph_0$ with an almost disjoint forcing, where we first reshape $A$ to a predicate $A'$. In this reshaping forcing we destroy the Replacement axiom level-by-level relative to $A$ and therefore it is also destroyed level-by-level relative to $X$. We arrive at a model $L_{\alpha^*}[X]$, with $X\subseteq\aleph_0$, which is the least transitive ZFC$^-$ model containing $X$ and from this we can go back to a minimal $\beta$-model of second-order arithmetic.
\end{proof}

Of course, definable hyperclass forcing is not the last step in considering 
a hierarchy of forcing notions via their size. One could ask further:\\

\emph{Question}: What would a general hyperclass forcing look like and in which context can it be developed (a hypercass theory)? What would a hyperhyperclass forcings look like, i.e. a forcing where conditions are hyperclasses? \\

Here we developed a further step in this hierarchy after set forcing, definable class forcing and class-forcing in MK. We hope that it will serve as a basis for further fruitful research.

\bibliographystyle{amsalpha} 
\bibliography{biblio-diss_antos.bib}

\end{document}